\documentclass[reqno,english]{amsart}

\usepackage{soul}

\usepackage[utf8]{inputenc}
\usepackage{url}
\usepackage{etoolbox}           
\usepackage{xargs}           	
\usepackage{xstring}           	
\usepackage{graphics}
\usepackage{graphicx}
\usepackage{amsmath}
\usepackage{amssymb}
\usepackage{mathrsfs}
\usepackage{mathtools}
\usepackage[noend]{algpseudocode}
\usepackage{cite}
\usepackage[usenames,dvipsnames]{xcolor}
\usepackage{enumerate}
\usepackage{comment}
\usepackage{multirow}
\usepackage[english]{babel}

\usepackage{enumitem}
\usepackage{ifpdf}
\usepackage{color}
\usepackage{breqn}
\usepackage{listings}
\usepackage{changepage}
\usepackage{array}
\usepackage{subcaption}

\usepackage[
	breaklinks,
	pdfpagelabels,
	bookmarksnumbered,
	bookmarks      = false,
	hypertexnames  = false,
	plainpages     = false,
	colorlinks     = true,
	linkcolor      = MidnightBlue,
	citecolor      = Green,
	hyperfootnotes = true]{hyperref}
\expandafter\let\csname ver@amsthm.sty\endcsname\relax
\let\theoremstyle\relax

\usepackage{amsthm}           
\usepackage{cleveref}         
\usepackage{algorithm}

\newcommandx\dirder[3][{3=\cdot}]{
	#1'({}#2{};{}#3{})
}
\newcommandx\jac[3][{3={}}]{
	J#1
	\ifstrempty{#2}{}{%
		\ifstrempty{#3}{({}#2{})}{%
			({}#2{})[{}#3{}]%
		}%
	}%
}
\newcommandx\Bjac[3][{3={}}]{
	D#1
	\ifstrempty{#2}{}{%
		\ifstrempty{#3}{({}#2{})}{%
			({}#2{})[{}#3{}]%
		}%
	}%
}
\newcommandx\twiceepi[4][{1={}},{4={}}]{
	{\rm d}^2#2
	\ifstrempty{#3}{}{
		\ifstrempty{#1}{{(}{#3}{)}}{
			{(}{#3}{|}{#1}{)}%
		}%
	}%
	\ifstrempty{#4}{}{
			{[}{#4}{]}%
	}%
}
\newcommand\DEF[1]{\emph{#1}}
\newcommandx\seq[2][{2=k\in\Nn}]{(#1)_{#2}}
\newcommand\proj[1]{P_{#1}}
\newcommand\set[2]{%
	\ifstrempty{#1}{
		\left\{#2\right\}
	}{
		\left\{#1{}\mid{}#2\right\}
	}
}
\newcommand{\T}{\top} 
\newcommand\cont{C}

\algdef{SE}{Loop}{EndLoop}{\textbf{loop}}{\textbf{end}}

\renewcommand{\Re}{{\rm{I\!R}} }
\newcommand{\Nn}{{\rm{I\!N}} }

\newcommand{\dom}{\mathop{\rm dom}\nolimits}

\newcommand{\relint}{\mathop{\rm relint}\nolimits}
\newcommand{\dist}{\mathop{\rm dist}\nolimits}
\newcommand{\diag}{\mathop{\rm diag}\nolimits}

\newcommand{\prox}{\mathop{\rm prox}\nolimits}

\DeclareMathOperator*{\minimize}{{\rm minimize}}

\newcommand{\Rinf}{\overline{\rm I\!R} }

\DeclareMathOperator*{\argmin}{\arg\!\min}

\newcommand{\innprod}[2]{\langle #1,#2 \rangle}
\newcommand{\zer}{\mathop{\rm zer}\nolimits}

\newcommand{\ie}{\emph{i.e.}}

\newcommand{\OurSpace}{\Re^n}

\definecolor{red}{rgb}{1,0,0}
\definecolor{yellow}{rgb}{1,1,0}
\definecolor{blue}{rgb}{0,0,1}

\algtext*{EndWhile}
\algtext*{EndIf}

\newcolumntype{L}[1]{>{\raggedright\let\newline\\\arraybackslash\hspace{0pt}}m{#1}}
\newcolumntype{C}[1]{>{\centering\let\newline\\\arraybackslash\hspace{0pt}}m{#1}}
\newcolumntype{R}[1]{>{\raggedleft\let\newline\\\arraybackslash\hspace{0pt}}m{#1}}

\makeatletter
	\@addtoreset{equation}{section}
\makeatother

\newlist{enumalg}{enumerate}{1}
\setlist[enumalg]{label=\textit{(C\arabic*)},ref=\thealgorithm(C\arabic*)}
\crefalias{enumalgi}{algorithm} 

\def\amiinappendix{0}
\newcommand\inappendixstr{%
	\expandafter\ifstrequal\expandafter{\amiinappendix}{1}{}{ in the Appendix}%
}%

\newcommandx\CreateNewTheorem[4][{2={}},{4={}}]{%
	\ifstrempty{#2}{%
		\ifstrempty{#4}{%
			\newtheorem{#1}{#3}%
			\newtheorem{appendix#1}{#3}%
		}{%
			\newtheorem{#1}{#3}[#4]%
			\newtheorem{appendix#1}{#3}[#4]%
		}%
	}{%
		\ifstrempty{#4}{%
			\newtheorem{#1}[#2]{#3}%
			\newtheorem{appendix#1}[#2]{#3}%
		}{%
			\newtheorem{#1}[#2]{#3}[#4]%
			\newtheorem{appendix#1}[#2]{#3}[#4]%
		}%
	}%
		\crefname{#1}{\MakeLowercase #3}{\MakeLowercase #3s}
		\Crefname{#1}{\MakeUppercase #3}{\MakeUppercase #3s}
		\crefformat{#1}{##2\MakeLowercase #3 ##1##3}
		\Crefformat{#1}{##2\MakeUppercase #3 ##1##3}
		\crefmultiformat{#1}{##2\MakeLowercase #3s ##1##3}{ and ##2##1##3}{, ##2##1##3}{ and ##2##1##3}
		\Crefmultiformat{#1}{##2\MakeUppercase #3s ##1##3}{ and ##2##1##3}{, ##2##1##3}{ and ##2##1##3}
		\newlist{enum#1}{enumerate}{1}
		\setlist[enum#1]{%
			{label=(\roman*)},%
			{ref=\expandafter\csname the#1\endcsname(\roman*)},%
			{itemsep=3pt},%
			{labelwidth=15pt},
			{labelsep=5pt},%
			{itemindent=0pt},%
			{leftmargin=20pt},
		}%
		\crefalias{enum#1i}{#1} 
		\crefname{appendix#1}{\MakeLowercase #3}{\MakeLowercase #3s}
		\Crefname{appendix#1}{\MakeUppercase #3}{\MakeUppercase #3s}
		\crefformat{appendix#1}{##2\MakeLowercase #3 ##1##3\inappendixstr{}}
		\Crefformat{appendix#1}{##2\MakeUppercase #3 ##1##3\inappendixstr{}}
		\crefmultiformat{appendix#1}{##2\MakeLowercase #3s ##1##3}{ and ##2##1##3}{, ##2##1##3}{ and ##2##1##3\inappendixstr{}}
		\Crefmultiformat{appendix#1}{##2\MakeUppercase #3s ##1##3}{ and ##2##1##3}{, ##2##1##3}{ and ##2##1##3\inappendixstr{}}
		\newlist{enumappendix#1}{enumerate}{1}
		\setlist[enumappendix#1]{label=(\roman*),ref=\expandafter\csname theappendix#1\endcsname(\roman*)}
		\crefalias{enumappendix#1i}{appendix#1} 
}

\theoremstyle{definition}
	\CreateNewTheorem{es}{Example}[section]
	
	\CreateNewTheorem{fact}[es]{Fact}
	\CreateNewTheorem{rem}[es]{Remark}

\theoremstyle{plain}
	\CreateNewTheorem{defin}[es]{Definition}
	\CreateNewTheorem{thm}[es]{Theorem}
	\CreateNewTheorem{prop}[es]{Proposition}
	\CreateNewTheorem{cor}[es]{Corollary}
	\CreateNewTheorem{lem}[es]{Lemma}
	\CreateNewTheorem{ass}{Assumption}

\AtBeginEnvironment{appendix}{%
	\def\amiinappendix{1}%
	\renewenvironment{es}[1][]{\begin{appendixes}[{#1}]}{\end{appendixes}}%
	\renewenvironment{rem}[1][]{\begin{appendixrem}[{#1}]}{\end{appendixrem}}%
	\renewenvironment{defin}[1][]{\begin{appendixdefin}[{#1}]}{\end{appendixdefin}}%
	\renewenvironment{thm}[1][]{\begin{appendixthm}[{#1}]}{\end{appendixthm}}%
	\renewenvironment{prop}[1][]{\begin{appendixprop}[{#1}]}{\end{appendixprop}}%
	\renewenvironment{cor}[1][]{\begin{appendixcor}[{#1}]}{\end{appendixcor}}%
	\renewenvironment{lem}[1][]{\begin{appendixlem}[{#1}]}{\end{appendixlem}}%
	\renewenvironment{ass}[1][]{\begin{appendixass}[{#1}]}{\end{appendixass}}%
}%

\renewcommand\theequation{\thesection.\arabic{equation}}
\makeatletter
\@addtoreset{paragraph}{subsection}
\makeatother

\crefformat{section}{#2§#1#3}
\Crefformat{section}{#2Section #1#3}
\crefmultiformat{section}{§#2#1#3}{ and §#2#1#3}{, §#2#1#3}{ and §#2#1#3}
\Crefmultiformat{section}{Sections #2#1#3}{ and #2#1#3}{, #2#1#3}{ and #2#1#3}

\crefformat{subsection}{#2§#1#3}
\Crefformat{subsection}{#2Section #1#3}
\crefmultiformat{subsection}{§#2#1#3}{ and §#2#1#3}{, §#2#1#3}{ and §#2#1#3}
\Crefmultiformat{subsection}{Sections #2#1#3}{ and #2#1#3}{, #2#1#3}{ and #2#1#3}

\crefformat{subsubsection}{#2§#1#3}
\Crefformat{subsubsection}{#2Section #1#3}
\crefmultiformat{subsubsection}{§#2#1#3}{ and §#2#1#3}{, §#2#1#3}{ and §#2#1#3}
\Crefmultiformat{subsubsection}{Sections #2#1#3}{ and #2#1#3}{, #2#1#3}{ and #2#1#3}

\crefformat{paragraph}{#2§#1#3}
\Crefformat{paragraph}{#2Section #1#3}
\crefmultiformat{paragraph}{§#2#1#3}{ and §#2#1#3}{, §#2#1#3}{ and §#2#1#3}
\Crefmultiformat{paragraph}{Sections #2#1#3}{ and #2#1#3}{, #2#1#3}{ and #2#1#3}

\crefformat{appendix}{#2§#1#3}
\Crefformat{appendix}{#2Appendix #1#3}
\crefmultiformat{appendix}{§#2#1#3}{ and §#2#1#3}{, §#2#1#3}{ and §#2#1#3}
\Crefmultiformat{appendix}{#2Appendices #1#3}{ and #2#1#3}{, #2#1#3}{ and #2#1#3}

\crefformat{enumi}{#2#1#3}
\Crefformat{enumi}{#2#1#3}
\crefmultiformat{enumi}{#2#1#3}{ and #2#1#3}{, #2#1#3}{ and #2#1#3}
\Crefmultiformat{enumi}{#2#1#3}{ and #2#1#3}{, #2#1#3}{ and #2#1#3}

\crefformat{equation}{(#2#1#3)}
\Crefformat{equation}{Equations (#2#1#3)}
\crefmultiformat{equation}{(#2#1#3)}{ and (#2#1#3)}{, (#2#1#3)}{ and (#2#1#3)}
\Crefmultiformat{equation}{(#2#1#3)}{ and (#2#1#3)}{, (#2#1#3)}{ and (#2#1#3)}
\labelcrefrangeformat{equation}{#3(#1)#4--#5(#2)#6}

\labelcrefformat{subequation}{#2(#1)#3}
\labelcrefrangeformat{subequation}{#3(#1)#4--#5(#2)#6}

\makeatletter
\crefrangelabelformat{equation}{%
  \protected@edef\eq@ref@a{#1}%
  \protected@edef\eq@ref@b{#2}%
  \expandafter\expandafter\expandafter
  \eq@ref@check\expandafter\eq@ref@a\expandafter.\expandafter\@nil
                           \eq@ref@b.\@nil
  (%
    #3\eq@ref@a#4%
  )%
   --%
  (%
    #5\eq@ref@b#6%
  )%
}
\def\eq@ref@check#1.#2\@nil#3.#4\@nil{%
  \def\eq@tmp{#2}%
  \ifx\eq@tmp\@empty
  \else
    \def\eq@tmp{#4}%
    \ifx\eq@tmp\@empty
    \else
      \def\eq@tmp@a{#1}%
      \def\eq@tmp@b{#3}%
      \ifx\eq@tmp@a\eq@tmp@b
        \expandafter\def\expandafter\eq@ref@b\expandafter{%
          \eq@strip@dot#4\@nil
        }%
      \fi 
    \fi
  \fi
}
\def\eq@strip@dot#1.\@nil{#1}
\makeatother

\crefformat{algorithm}{#2Alg. #1#3}
\Crefformat{algorithm}{#2Algorithm #1#3}
\crefmultiformat{algorithm}{#2Alg. #1#3}{ and #2#1#3}{, #2#1#3}{ and #2#1#3}
\Crefmultiformat{algorithm}{#2Algorithms #1#3}{ and #2#1#3}{, #2#1#3}{ and #2#1#3}

\crefformat{ALC@unique}{#2step #1#3}
\Crefformat{ALC@unique}{#2Step #1#3}
\crefmultiformat{ALC@unique}{#2steps #1#3}{ and #2#1#3}{, #2#1#3}{ and #2#1#3}
\Crefmultiformat{ALC@unique}{#2Steps #1#3}{ and #2#1#3}{, #2#1#3}{ and #2#1#3}
\crefrangelabelformat{ALC@unique}{#3#1#4--#5\crefstripprefix{#1}{#2}#6}

\crefformat{figure}{#2Fig. #1#3}
\Crefformat{figure}{#2Figure #1#3}
\crefmultiformat{figure}{#2Fig. #1#3}{ and #2#1#3}{, #2#1#3}{ and #2#1#3}
\Crefmultiformat{figure}{#2Figures #1#3}{ and #2#1#3}{, #2#1#3}{ and #2#1#3}

\crefformat{table}{#2Tbl. #1#3}
\Crefformat{table}{#2Table #1#3}
\crefmultiformat{table}{#2Fig. #1#3}{ and #2#1#3}{, #2#1#3}{ and #2#1#3}
\Crefmultiformat{table}{#2Tables #1#3}{ and #2#1#3}{, #2#1#3}{ and #2#1#3}

\title[F-B quasi-Newton methods for nonsmooth optimization problems]
    {Forward-backward quasi-Newton methods for nonsmooth optimization problems}
\author{Lorenzo Stella, Andreas Themelis, Panagiotis Patrinos}
\date{\today}

\begin{document}

\begin{abstract}
	The forward-backward splitting method (FBS) for minimizing a nonsmooth composite function
can be interpreted as a (variable-metric) gradient method over a continuously differentiable function which we call
forward-backward envelope (FBE). This allows to extend algorithms for smooth unconstrained
optimization and apply them to nonsmooth (possibly constrained) problems.
Since the FBE and its gradient can be computed by simply evaluating forward-backward steps,
the resulting methods rely on the very same black-box oracle as FBS.
We propose an algorithmic scheme that enjoys the same global convergence properties of FBS when the problem is convex,
or when the objective function possesses the Kurdyka-{\L}ojasiewicz property at its critical points.
Moreover, when using quasi-Newton directions the proposed method achieves superlinear convergence
provided that usual second-order sufficiency conditions on the FBE hold at the limit point of the generated sequence.
Such conditions translate into milder requirements on the original function involving generalized second-order differentiability.
We show that BFGS fits our framework and that the limited-memory variant L-BFGS is well suited for large-scale problems,
greatly outperforming FBS or its accelerated version in practice.
The analysis of superlinear convergence is based on an extension of the Dennis and Mor\'e theorem for the proposed algorithmic scheme.

\end{abstract}

\maketitle


	\section{Introduction}
		\label{SEC:Introduction}
		In this paper we focus on nonsmooth optimization problems over $\OurSpace$
of the form
\begin{equation}\minimize_{x\in \OurSpace}\ \varphi(x) \equiv f(x) + g(x),
\label{eq:GenProb}\end{equation}
where $f$ is a smooth (possibly nonconvex) function,
while $g$ is a proper, closed, convex (possibly nonsmooth) function with
cheaply computable proximal mapping~\cite{moreau1965proximiteet}. Problems of
this form appear in several application fields such as control, system identification,
signal and image processing, machine learning and statistics.

Perhaps the most well known algorithm to solve problem~\eqref{eq:GenProb} is the
forward-backward splitting (FBS), also known as proximal gradient method
\cite{lions1979splitting, combettes2011proximal}, which generalizes the classical gradient
method to problems involving an additional nonsmooth term. Convergence of the iterates
of FBS to a critical point of problem~\eqref{eq:GenProb} has been shown, in the general
nonconvex case, for functions $\varphi$ having the Kurdyka-{\L}ojasiewicz property
\cite{lojasiewicz1963propriete,lojasiewicz1993geometrie,kurdyka1998gradients,Attouch2013}.
This assumption was used to prove convergence of many other
algorithms \cite{Attouch2009,Attouch2010,Attouch2013,bolte2014proximal,Ochs2014}.
The global convergence rate of FBS is known to be sublinear of order $O(1/k)$ in the convex case,
where $k$ is the iteration count, and can be improved to $O(1/k^2)$ with techniques based on the
work of Nesterov \cite{Nesterov1983,tseng2008accelerated,beck2009fast,nesterov2013gradient}.
Therefore, FBS is usually efficient for computing solutions with small to medium precision only
and, just like all first order methods, suffers from ill-conditioning of the problem at hand. A remedy to this
is to add second-order information in the computation of the forward and backward steps, so to
better scale the problem and achieve superlinear asymptotic convergence. As proposed by several authors
\cite{becker2012quasi,lee2012proximal,scheinberg2013},
this can be done by computing the gradient steps and proximal steps according to the $Q$-norm
rather than the Euclidean norm, where $Q$ is the Hessian of $f$ or some approximation to it.
This approach has the severe limitation that, unless $Q$ has a very particular structure,
the backward step becomes now very hard and requires an inner iterative procedure
to be computed.

In the present paper we follow a different approach. We define a function, which we call
\emph{forward-backward envelope} (FBE) that serves as a real-valued, continuously differentiable, exact penalty function
for the original problem. Furthermore, forward-backward splitting is shown to be equivalent to
a (variable-metric) gradient method applied to the problem of minimizing the FBE. The value and gradient of the
FBE can be computed solely based on the evaluation of a forward-backward step at the point of interest.
For these reasons, the FBE works as a surrogate of the Moreau envelope \cite{moreau1965proximiteet}
for composite problems of the form~\eqref{eq:GenProb}.
Most importantly, this opens up the possibility of using well known smooth unconstrained optimization
algorithms, with faster asymptotic convergence properties than the gradient method, to minimize the
FBE and thus solve \eqref{eq:GenProb}, which is nonsmooth and possibly constrained.
This approach was first explored in \cite{patrinos2013proximal}, where two Newton-type method were proposed,
and combines and extends ideas stemming from the literature on merit functions for
\emph{variational inequalities} (VIs) and \emph{complementarity problems} (CPs),
specifically the reformulation of a VI as a constrained continuously differentiable optimization problem
via the regularized gap function \cite{fukushima1992equivalent}
and as an unconstrained continuously differentiable optimization problem via the D-gap function \cite{yamashita1997unconstrained}
(see \cite[§10]{facchinei2003finite} for a survey and \cite{Li2007exact},
\cite{patrinos2011global} for applications to constrained optimization and model predictive control of dynamical systems).

Then we propose an algorithmic scheme, based on line-search methods, to minimize the FBE.
In particular, when descent steps are taken along quasi-Newton directions, superlinear
convergence can be achieved when usual nonsingularity assumptions hold at the limit point of the sequence
of iterates. The asymptotic analysis is based on an analogous of the
Dennis and Mor\'e theorem \cite{dennis1974characterization} for the proposed algorithmic scheme,
and the BFGS quasi-Newton method is shown to fit this framework. Its limited memory variant
L-BFGS, which is suited for large scale problems, is also analyzed.
At the same time, we show that our algorithm enjoys the same global convergence properties of FBS under the same assumptions on the original function \(\varphi\), despite our method operates on the surrogate \(\varphi_\gamma\).
Unlike the approaches of \cite{becker2012quasi,lee2012proximal,scheinberg2013}, our algorithm is based on the very same
forward-backward operations as FBS, and does not require the solution to any inner problem.

The contributions of this work can be summarized as follows. We give an interpretation
of forward-backward splitting as a (variable-metric) gradient method over a continuously differentiable function,
the forward-backward envelope (FBE). Then we propose an algorithmic scheme for solving problem
\eqref{eq:GenProb} based on line-search methods applied to the problem of minimizing the FBE,
and prove that it converges globally to a critical point when $\varphi$ is convex or has the
Kurdyka-{\L}ojasiewicz property. This is a crucial feature of our approach: in fact, the FBE is nonconvex in general, and
there exist examples showing how classical line-search methods need not converge to critical points
for nonconvex functions \cite{Dai2002,Mascarenhas2004,Mascarenhas2007,Dai2013}.
When $\varphi$ is convex, in addition, global sublinear convergence of order $O(1/k)$ (in the objective value) is proved.
Finally, we show that when the directions of choice satisfy the Dennis-Mor\'e condition, then the
method converges superlinearly, under appropriate assumptions, and illustrate when this is the
case for BFGS.

The paper is organized as follows.
\Cref{SEC:FBE} introduces the forward-backward
envelope function and illustrates its properties.
In \Cref{SEC:Algorithms} we propose our algorithmic scheme and prove 
its global convergence properties. Linear convergence is also discussed.
\Cref{SEC:QuasiNewton} is devoted to the asymptotic convergence analysis in the particular case where,
in the proposed method, quasi-Newton directions are used and specialize this result to the case of BFGS.
Limited-memory directions are also discussed.
Finally, \Cref{SEC:Simulations} illustrates numerical results obtained with the proposed method.
Some of the proofs are deferred to the Appendix for the sake of readability.
Moreover, for the reader's convenience, \Cref{APP:Definitions} will list some definitions and known results on generalized differentiability which are needed in the analysis.

		\subsection{Notation and background}
			Throughout the paper, $\innprod{{}\cdot{}}{{}\cdot{}}$ is an inner product over $\OurSpace$ and $\|{}\cdot{}\|=\sqrt{\innprod{{}\cdot{}}{{}\cdot{}}}$ is the induced norm.
The set of continuously differentiable functions
on $\Re^n$ having $L$-Lipschitz continuous gradient (also refferred to as $L$-smooth) is denoted by $\cont_{L}^{1,1}(\OurSpace)$. We denote
the extended real line as $\Rinf \equiv \Re \cup \set{}{+\infty}$.
The set of proper, closed, convex functions from $\OurSpace$ with values in $\Rinf$ is referred to as $\Gamma_0(\OurSpace)$.
Given a function $h$ on $\OurSpace$,
the subdifferential $\partial h(x)$ of $h$ at $x$ is considered in the sense
of \cite[Def.~8.3]{rockafellar2011variational}, that is
\begin{align*}
	\partial h(x)
{}={} &
	\set{v\in\Re^n}{
		\exists\ \seq{x^k}, \seq{v^k\in\hat\partial h(x^k)} ~\text{s.t.}~ x^k\to x, v^k\to v
	}
\intertext{where}
	\hat\partial h(x)
{}={} &
	\set{v\in\Re^n}{
		h(z)\geq h(x) + \innprod{v}{z-x} + o(\|z-x\|), \forall z\in\Re^n
	}.
\end{align*}
This includes the ordinary gradient in the case of continuously differentiable functions, while for $g\in\Gamma_0(\OurSpace)$ it is equivalent to
\[
	\partial g(x)
{}={}
	\set{v\in\Re^n}{
		g(y) \geq g(x) + \innprod{v}{y-x},\ \mathrm{for\ all}\ y\in \OurSpace
	}.
\]
We denote the set of \emph{critical points} associated with problem \eqref{eq:GenProb} as
\[
	\zer\partial\varphi
{}={}
	\set{x\in\Re^n}{0\in\partial \varphi(x)}
{}={}
	\set{x\in\Re^n}{-\nabla f(x)\in\partial g(x)}.
\]
The second equality is due to \cite[Ex.~8.8]{rockafellar2011variational}. A necessary condition
for a point $x$ to be a local minimizer for \eqref{eq:GenProb} is that $x\in\zer\partial\varphi$ \cite[Thm.~10.1]{rockafellar2011variational}.
If $\varphi$ is convex (for example when $f$ is convex) then the condition is also sufficient, and $x$ is a global minimizer.

Given $g\in\Gamma_0(\Re^n)$,
its \emph{proximal mapping} is defined by
\begin{equation}\label{eq:prox}
	\prox_{\gamma g}(x)
{}={}
	\argmin_{u\in \OurSpace}{
		\set{}{g(u)+\tfrac{1}{2\gamma}\|u-x\|^2}
	},
\end{equation}
cf.~\cite{moreau1965proximiteet}.
The proximal mapping is a generalized projection, in the sense that if \(g=\delta_C\) is the \emph{indicator function} of a nonempty closed convex set \(C\subseteq\Re^n\), \ie, \(g(x)=0\) for \(x\in C\) and \(+\infty\) otherwise, then \(\prox_{\gamma g}=\proj{C}\) is the projection on \(C\) for any \(\gamma>0\).
The value function of the optimization problem~\eqref{eq:prox} defining the
proximal mapping is called the \emph{Moreau envelope} and is denoted by
$g^\gamma$, \ie,
\begin{equation}\label{eq:MoreauEnv}
g^{\gamma}(x) = \min_{u\in \OurSpace}\set{}{g(u)+\tfrac{1}{2\gamma}\|u-x\|^2}.
\end{equation}
Properties of the Moreau envelope and the proximal mapping are well documented
in the literature \cite{bauschke2011convex,rockafellar2011variational,
combettes2005signal,combettes2011proximal}. For example, the proximal mapping is
single-valued, continuous and nonexpansive (Lipschitz continuous with Lipschitz constant
$1$) and the envelope function $g^{\gamma}$ is convex, continuously
differentiable, with gradient
\begin{equation}\label{eq:nabla_e}
\nabla g^{\gamma}(x) = \gamma^{-1}(x-\prox_{\gamma g}(x)),
\end{equation}
which is $\gamma^{-1}$-Lipschitz continuous.

We will consider cases where $g$ is \emph{twice epi-differentiable} \cite[Def.~13.6]{rockafellar2011variational},
and indicate with $\twiceepi[v]{g}{x}$ the second-order epi-derivative of $g$ at $x$ for $v$.

For a mapping $F:\Re^n\to\Re^m$ we will indicate by $\Bjac{F}{x}$ and $\jac{F}{x}$, respectively, its semiderivative
and Jacobian at $x$, when these exist. The directional derivative of $F$ at $x$ along a direction $d$ will then
be denoted as $\Bjac{F}{x}[d]$ if $F$ is semidifferentiable at $x$, and as $\jac{F}{x}[d] = \jac{F}{x} d$ if $F$ is differentiable at $x$.
For the basic notions about semidifferentiability, and its link with ordinary differentiability, we refer the reader to \Cref{APP:Definitions} and the references therein.

		\subsection{The forward-backward splitting}
			In the rest of the paper we will work under the following
\begin{ass}\label{Ass:fg}
\(\varphi=f+g\) with \(f\in \cont^{1,1}_{L_f}(\Re^n)\) for some \(L_f>0\) and \(g\in\Gamma_0(\Re^n)\).
\end{ass}
If \(f\) satisfies \Cref{Ass:fg} then \cite[Prop. A.24]{bertsekas1999nonlinear}
\begin{equation}\label{Eq:DescentLemma}
	f(y)
{}\leq{}
	f(x)
	{}+{}
	\innprod{\nabla f(x)}{y-x}
	{}+{}
	\tfrac{L_f}{2}
	\|y-x\|^2.
\end{equation}
Given an initial point $x^0$ and $\gamma>0$, forward-backward splitting (also known as proximal gradient method) seeks solutions to the problem \eqref{eq:GenProb} by means of the following iterations:
\begin{equation}\label{eq:FBS}
x^{k+1} = \prox_{\gamma g}(x^k-\gamma\nabla f(x^k)).
\end{equation}
Under \Cref{Ass:fg} the generated sequence \(\seq{x^k}\) satisfies \cite[eq. (2.13)]{nesterov2013gradient}
\[
	\varphi(x^{k+1})-\varphi(x^k)
{}\leq{}
	-\tfrac{2-\gamma L_f}{2\gamma}
	\|x^{k+1}-x^k\|^2.
\]
If $\gamma\in(0,2/L_f)$ and \(\varphi\) is lower bounded, it can be easily inferred that any cluster point \(x\) is stationary for \(\varphi\), in the sense that it satisfies the necessary condition for optimality \(x\in\zer\partial\varphi\).
The existence of cluster points is ensured if \(\seq{x^k}\) remains bounded; due to the monotonic behavior of \(\seq{\varphi(x^k)}\) for \(\gamma\) in the given range, this condition in turn is guaranteed if \(\varphi\) and the initial point \(x^0\) satisfy the following requirement, which is a standard assumption for nonconvex problems (see e.g. \cite{nesterov2013gradient}).
\begin{ass}\label{Ass:LevelSets}
The level set \(\set{x\in\Re^n}{\varphi(x)\leq\varphi(x^0)}\), which for conciseness we shall denote \(\set{}{\varphi\leq\varphi(x^0)}\), is bounded.
In particular, there exists $R>0$ such that
\(
	\|x-z\|
{}\leq{}
	R
\)
for all $x\in\set{}{\varphi\leq\varphi(x^0)}$ and $z\in\argmin\varphi$.
\end{ass}
The existence of such a uniform radius $R$ is due to boundedness of $\argmin\varphi$, which in turn follows from the assumed boundedness of $\set{}{\varphi\leq\varphi(x^0)}$.

\begin{es}
To see that \(\argmin\varphi\neq\emptyset\) is not enough for preventing the generation of unbounded sequences, consider \(\varphi=f+g:\Re\to\Rinf\) where
\[
	g=\delta_{(-\infty,2]}
\quad\text{and}\quad
	f(x)
{}={}
	\begin{cases}
		\exp(x)-1 & \text{if }x<0, \\
		x-x^2 & \text{if }x\geq 0.
	\end{cases}
\]
\Cref{Ass:fg} is satisfied with \(L_f=2\) and \(\argmin\varphi=\set{}{2}\).
However, for any \(\gamma\in(0,1)\) the sequence \(\seq{x^k}\) generated by \eqref{eq:FBS} with \(x^0<1/2\) diverges to \(-\infty\), and \(\varphi(x^k)\to -1>-2=\min\varphi\).
This however cannot happen in the convex case \cite[Thm. 25.8]{bauschke2011convex}.
\end{es}

We use shorthands to denote the forward-backward mapping and the associated \emph{fixed-point residual}
in order to simplify the notation:
\begin{align}\label{eq:Shorthands}
T_\gamma(x) &= \prox_{\gamma g}(x-\gamma\nabla f(x)), \\
R_\gamma(x) &= \gamma^{-1}(x-T_\gamma(x)),
\end{align}
so that iteration~\eqref{eq:FBS} can be written as $x^{k+1} = T_\gamma(x^k) = x^k - \gamma R_\gamma(x^k)$.
The set $\zer\partial\varphi$ is easily characterized in terms of the fixed-point set of $T_\gamma$ as follows:
\begin{equation}
	x=T_\gamma(x)
{}\Longleftrightarrow{}
	x\in\zer\partial\varphi.\label{eq:StatCond}
\end{equation}

Note that $T_\gamma(x)$ can alternatively be expressed as the solution to the following
partially linearized subproblem (see also \Cref{fig:FBS}):
\begin{subequations}
\begin{align}
	T_\gamma(x)
{}={} &
	\argmin_{u\in \OurSpace}{
		\set{}{\ell_\varphi(u, x) + \tfrac{1}{2\gamma}\|u-x\|^2},
	}
\label{eq:ProxGradProblem}
\\
	\ell_\varphi(u, x)
{}={}&
	f(x) + \innprod{\nabla f(x)}{u-x} + g(u).
\label{eq:Linearization}
\end{align}
\end{subequations}

	\section{Forward-backward envelope}
		\label{SEC:FBE}
		We now proceed to the reformulation of~\eqref{eq:GenProb} as the
minimization of an unconstrained continuously differentiable function.
To this end, we consider the value function of problem~\eqref{eq:ProxGradProblem} defining the forward-backward mapping $T_\gamma$
and give the following definition.
\begin{defin}[Forward-backward envelope]\label{def:FBE}
Let \(f,g\) and \(\varphi\) be as in \Cref{Ass:fg}, and let $\gamma>0$.
The forward-backward envelope (FBE) of $\varphi$ with parameter $\gamma$ is
\begin{equation}
	\varphi_\gamma(x)
{}={}
	\min_{u\in \OurSpace}{
		\set{}{\ell_\varphi(u, x) + \tfrac{1}{2\gamma}\|u-x\|^2}.
	}
\label{eq:FBEmin}
\end{equation}
\end{defin}
Using \eqref{eq:ProxGradProblem} and \eqref{eq:Linearization} it is easy to verify that \eqref{eq:FBEmin} can be equivalently expressed as
\begin{align}
	\varphi_\gamma(x)
{}={} &
	f(x)
	{}+{}
	g(T_{\gamma}(x))
	{}-{}
	\gamma\innprod{\nabla f(x)}{R_\gamma(x)}
	{}+{}
	\tfrac{\gamma}{2}\|R_\gamma(x)\|^2\label{eq:FBEexp}
\intertext{or, by the definition of Moreau envelope, as}
	\varphi_\gamma(x)
{}={} &
	f(x)
	{}-{}
	\tfrac{\gamma}{2}\|\nabla f(x)\|^2
	{}+{}
	g^{\gamma}(x-\gamma\nabla f(x)).\label{eq:EquivFBE}
\end{align}
The geometrical construction of $\varphi_\gamma$ is depicted in \Cref{fig:FBE}.
One distinctive feature of $\varphi_\gamma$ is the fact that it is real-valued, despite the fact that $\varphi$
can be extended-real-valued.
Function $\varphi_\gamma$ has other favorable properties which we now summarize.

\begin{figure}[tb]
	\includegraphics{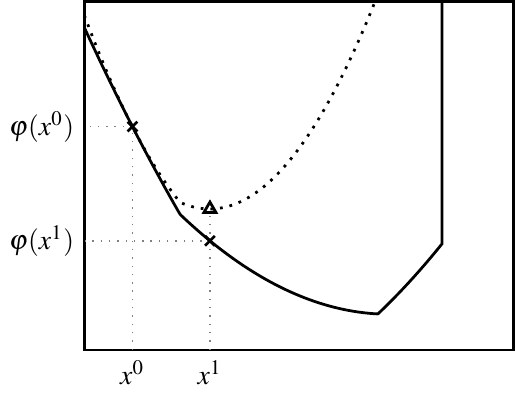}
	\hfill
	\includegraphics{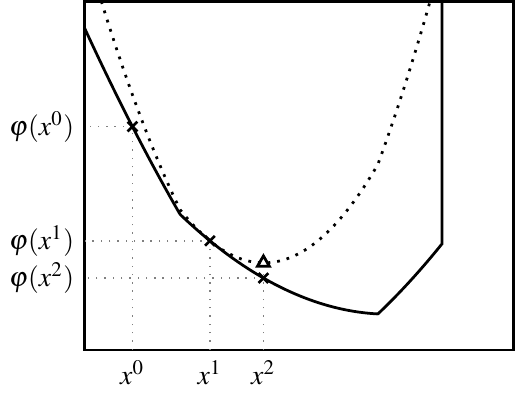}
	\caption{When $\gamma$ is small enough forward-backward splitting minimizes, at every step,
	a convex majorization (dotted lines) of the original cost $\varphi$ (solid line),
	cf.~\eqref{eq:ProxGradProblem}.}
	\label{fig:FBS}
\end{figure}

\begin{figure}[tb]
	\includegraphics{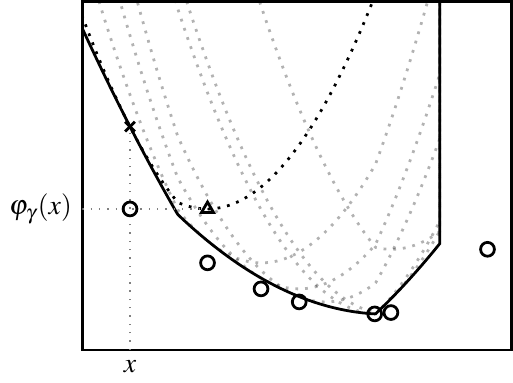}
	\hfill
	\includegraphics{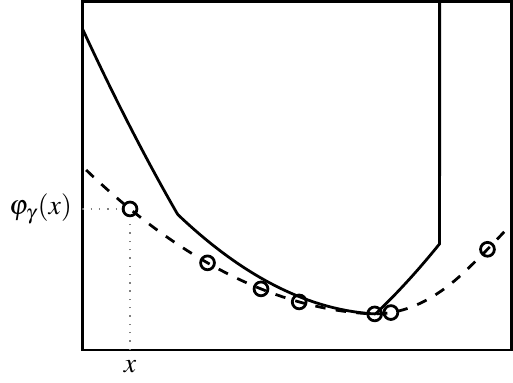}
	\caption{The forward-backward envelope $\varphi_\gamma$ (dashed line) is obtained by considering the optimal values of problems
    	\eqref{eq:ProxGradProblem} (dotted lines), and serves as a real-valued lower bound for the original objective $\varphi$ (solid line).}
    \label{fig:FBE}
\end{figure}

		\subsection{Basic inequalities}
			The following result states the fundamental inequalities relating $\varphi_\gamma$
to $\varphi$.

\begin{prop}\label{prop:BoundsFBE}
Suppose \Cref{Ass:fg} is satisfied.
Then, for all \(x\in\Re^n\)
\begin{enumprop}
	\item\label{prop:UppBnd}
		\(
			\displaystyle
			\varphi_\gamma(x)\leq \varphi(x)-\tfrac{\gamma}{2}\|R_\gamma(x)\|^2
		\)~
		for all $\gamma>0$;
	\item\label{prop:LowBnd}
		\(
			\displaystyle
			\varphi(T_\gamma(x)) \leq \varphi_\gamma(x)-\tfrac{\gamma}{2}
			\left(1-{\gamma}L_f\right)\|R_\gamma(x)\|^2
		\)~
		for all $\gamma>0$;
	\item\label{prop:LowBnd4Gamma}
		\(
			\displaystyle
			\varphi(T_\gamma(x))\leq \varphi_\gamma(x)
		\)~
		for all $\gamma\in(0,1/L_f]$.
\end{enumprop}
\end{prop}
\begin{proof}
Regarding \ref{prop:UppBnd}, from the optimality condition for \eqref{eq:ProxGradProblem} we have
$$ R_\gamma(x)-\nabla f(x)\in\partial g(T_\gamma(x)), $$
\ie, $R_\gamma(x)-\nabla f(x)$ is a subgradient of $g$ at $T_\gamma(x)$.
From subgradient inequality
\begin{align*}
g(x) &\geq g(T_\gamma(x))+\innprod{R_\gamma(x)-\nabla f(x)}{x-T_\gamma(x)}\\
 &= g(T_\gamma(x)) - \gamma\innprod{\nabla f(x)}{R_\gamma(x)} + \gamma\|R_\gamma(x)\|^2.
\end{align*}
Adding $f(x)$ to both sides and considering \eqref{eq:FBEexp} proves the claim. For \ref{prop:LowBnd}, we have
\begin{align*}
\varphi_\gamma (x) &= f(x) + \gamma\innprod{\nabla f(x)}{R_\gamma(x)} + g(T_{\gamma}(x)) + \tfrac{\gamma}{2}\|R_\gamma(x)\|^2\\
 &\geq f(T_\gamma(x)) + g(T_{\gamma}(x)) - \tfrac{L_f}{2}\|T_{\gamma}(x)-x\|^2 + \tfrac{\gamma}{2}\|R_\gamma(x)\|^2.
\end{align*}
where the inequality follows by \eqref{Eq:DescentLemma}.
\ref{prop:LowBnd4Gamma} then trivially follows.
\end{proof}

\begin{figure}[tb]
	\includegraphics{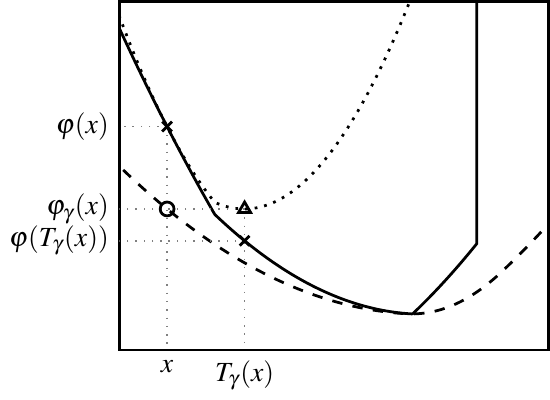}
	\hfill
	\includegraphics{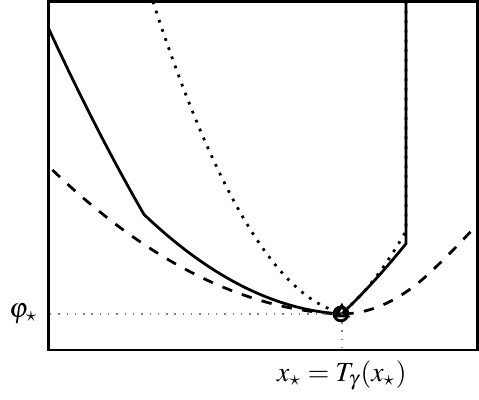}
	\caption{Left: by \Cref{prop:BoundsFBE}, $\varphi_\gamma(x)$ is upper bounded by $\varphi(x)$ and, when $\gamma$ is small enough, lower bounded by $\varphi(T_\gamma(x))$.
		Right: by \Cref{prop:FBEzer}, the two bounds coincide in correspondence of critical points.}
	\label{fig:FBEbounds}
\end{figure}

A consequence of \Cref{prop:BoundsFBE} is that, whenever $\gamma$ is
small enough, the problems of minimizing $\varphi$ and $\varphi_\gamma$ are equivalent.

\begin{prop}\label{prop:EquivFBE}
Suppose \Cref{Ass:fg} is satisfied.
Then,
\begin{enumprop}
	\item\label{prop:FBEzer}
		\(\varphi(z)=\varphi_\gamma(z)\)~ for all \(\gamma>0\) and \(z\in\zer\partial\varphi\);
	\item\label{prop:FBEinf}
		\(\inf\varphi=\inf\varphi_\gamma\)~ and ~\(\argmin \varphi\subseteq\argmin\varphi_\gamma\)~ for \(\gamma\in(0,1/L_f]\);
	\item\label{prop:FBEargmin}
		\(\argmin\varphi=\argmin\varphi_\gamma\)~ for all \(\gamma\in(0,1/L_f)\).
\end{enumprop}
\begin{proof}
\ref{prop:FBEzer} follows from \eqref{eq:StatCond}, \Cref{{prop:UppBnd},{prop:LowBnd}}.
In particular, this holds for $x_\star\in\argmin\varphi$.
Suppose now \(\gamma\in(0,1/L_f]\), then
$$\varphi_\gamma(x_\star) = \varphi(x_\star) \leq \varphi(T_\gamma(x)) \leq \varphi_\gamma(x) \qquad\text{for all } x \in \Re^n$$
where the first inequality follows from optimality of $x_\star$ for $\varphi$, and the second
from \Cref{prop:LowBnd4Gamma}.
Therefore, every $x_\star\in \argmin\varphi$ is also a minimizer of $\varphi_\gamma$, and $\min\varphi = \min\varphi_\gamma$ if the former is attained, which proves \ref{prop:FBEinf}.

Suppose now \(\gamma\in(0,1/L_f)\), and let $x_\star \in\argmin \varphi_\gamma$.
From \Cref{{prop:UppBnd},{prop:LowBnd}}
we get that
$$ \varphi_\gamma(T_\gamma(x_\star)) \leq \varphi(T_\gamma(x_\star)) \leq \varphi_\gamma(x_\star) - \tfrac{1-\gamma L_f}{2}\|x_\star - T_\gamma(x_\star)\|^2, $$
which implies $x_\star = T_\gamma(x_\star)$, since $x_\star$ minimizes $\varphi_\gamma$ and \(\tfrac{1-\gamma L_f}{2}>0\).
Therefore, the following chain of inequalities holds
$$ \varphi_\gamma(x_\star) = \varphi_\gamma(T_\gamma(x_\star)) \leq \varphi(x_\star) \leq \varphi_\gamma(x_\star). $$
Since $\varphi_\gamma \leq \varphi$ and $x_\star$ minimizes $\varphi_\gamma$, it follows that $x_\star\in\argmin \varphi$.
Therefore, the sets of minimizers of $\varphi$ and $\varphi_\gamma$ coincide, as well as the minimum values provided either one of the two functions attains it.

To conclude, it remains to show the case \(\argmin\varphi=\emptyset\) which, as just proven, corresponds to \(\argmin\varphi_\gamma=\emptyset\).
From \Cref{prop:UppBnd} we have \(\inf\varphi_\gamma\leq\inf\varphi\).
Suppose that the inequality is strict, \ie, that there exists \(x\in\Re^n\) such that \(\varphi_\gamma(x)\leq\inf\varphi\).
Then, \Cref{prop:LowBnd} would imply \(\varphi(T_\gamma(x))\leq\inf\varphi\), which contradicts \(\argmin\varphi=\emptyset\), and \ref{prop:FBEargmin} follows.
\end{proof}
\end{prop}
\begin{es}
To see that the bounds on \(\gamma\) in \Cref{prop:EquivFBE} are tight, consider the convex problem
\[
	\minimize_{x\in\Re^n}{
		\varphi(x)
	{}\equiv{}
		\smash{
			\overbracket[0.5pt]{\tfrac12\|x\|^2}^{f(x)}
			{}+{}
			\overbracket[0.5pt]{\vphantom{\tfrac12}\delta_{\Re_+^n}(x)}^{g(x)}
		}
	}
\]
where \(\Re_+^n=\set{x\in\Re^n}{x_i\geq 0, i=1\ldots n}\) is the nonnegative orthant.
\Cref{Ass:fg} is satisfied with \(L_f=1\), and the only stationary point for \(\varphi\) is the unique minimizer \(x_\star=0\).
Using \eqref{eq:EquivFBE} we can explicitly compute the FBE: for any \(\gamma>0\) we have
\[
	\varphi_\gamma(x)
{}={}
	\tfrac{1-\gamma}{2}\|x\|^2
	{}+{}
	\tfrac{1}{2\gamma}
	\bigl\|
		(1-\gamma)x
		{}-{}
		[(1-\gamma)x]_+
	\bigr\|^2,
\]
where \([x]_+=\Pi_{\Re_+^n}(x)=\max\set{}{x,0}\), the last expression being meant componentwise.
For any \(\gamma>0\) we have that \(\varphi_\gamma(x_\star)=\varphi(x_\star)\), as ensured by \Cref{prop:FBEzer}, and as long as \(\gamma<1=1/L_f\) all properties in \Cref{prop:EquivFBE} do hold.
For \(\gamma=1\) we have that \(\varphi_\gamma\equiv 0\), showing the inclusion in \Cref{prop:FBEinf} to be proper, yet satisfying \(\min\varphi_\gamma=\min\varphi\).

However, for \(\gamma>1\) the FBE \(\varphi_\gamma\) is not even lower bounded, as it can be easily deduced by observing that, letting \(x^k=(-k,0\ldots0)\) for \(k\in\Nn\), \(\varphi_\gamma(x^k)=\frac{1-\gamma}{2}k^2\) is arbitrarily negative.
\qed
\end{es}

\Cref{prop:EquivFBE} implies, using \Cref{prop:UppBnd}, that an $\varepsilon$-optimal solution $x$ of $\varphi$ is automatically $\varepsilon$-optimal for $\varphi_\gamma$
and, using \Cref{prop:LowBnd}, from an $\varepsilon$-optimal for $\varphi_\gamma$ we can directly obtain an $\varepsilon$-optimal solution
for $\varphi$ if $\gamma\in(0,1/L_f]$:
\begin{align*}
	\varphi(x)-\inf\varphi
{}\leq{} &
	\varepsilon\implies \varphi_\gamma(x)-\inf\varphi\leq\varepsilon
\\
	\varphi_\gamma(x)-\inf\varphi_\gamma
{}\leq{} &
	\varepsilon\implies \varphi(T_\gamma(x))-\inf\varphi\leq\varepsilon
\end{align*}
\Cref{prop:EquivFBE} also highlights the first apparent similarity between the concepts of FBE and Moreau envelope \eqref{eq:MoreauEnv}:
the latter is indeed itself a lower bound for the original function, sharing with it its minimizers and minimum value.
In fact, the two are directly related as we now show.
In particular, the following result implies that if $\varphi$ is convex (e.g. if $f$ is) and $\gamma\in(0, 1/L_f)$, then the possibly nonconvex $\varphi_\gamma$ is upper and lower bounded by convex functions.

\begin{prop}\label{prop:FBEvsMoreau}
Suppose \Cref{Ass:fg} is satisfied.
Then,
\begin{enumprop}
	\item\label{it:bndUp}
		$\displaystyle\varphi_\gamma\leq\varphi^{\frac{\gamma}{1+\gamma L_f}}$~ for all \(\gamma>0\);
	\item\label{it:bndLo}
		$\displaystyle\varphi^{\frac{\gamma}{1-\gamma L_f}}\leq\varphi_\gamma$~ for all $\gamma\in(0, 1/L_f)$;
	\item\label{it:bndUpCvx}
		$\displaystyle\varphi_\gamma\leq\varphi^\gamma$~ if $f$ is convex.
\end{enumprop}
\end{prop}
\begin{proof}
\eqref{Eq:DescentLemma} implies the following bounds concerning
the partial linearization:
\[
-\tfrac{L_f}{2}\|u-x\|^2\leq\varphi(u)-\ell_\varphi(u,x)\leq\tfrac{L_f}{2}\|u-x\|^2.
\]
Combined with the definition of the FBE, cf.~\eqref{eq:FBEmin}, this proves \ref{it:bndUp} and \ref{it:bndLo}.

If $f$ is convex, the lower bound can be strengthened to
\(
0\leq\varphi(u)-\ell_\varphi(u,x).
\)
Adding $\tfrac{1}{2\gamma}\|u-x\|^2$ to both sides and minimizing with respect to $u$ yields \ref{it:bndUpCvx}.
\end{proof}

		\subsection{Differentiability}
			We now turn our attention to differentiability of $\varphi_\gamma$, which is fundamental in devising and analyzing algorithms
for solving \eqref{eq:GenProb}.
To ensure continuous differentiability of \(\varphi_\gamma\) we will need the following
\begin{ass}\label{Ass:fC2}
The function \(f\) is twice-continuously differentiable over \(\Re^n\).
\end{ass}
Under \Cref{Ass:fC2}, the function
\begin{equation}\label{Eq:Q}
	Q_\gamma:\Re^n\to\Re^{n\times n}
\qquad\text{given by}\qquad
	Q_\gamma(x)
{}={}
	I-\gamma\nabla^2 f(x)
\end{equation}
is well defined, continuous, and symmetric-valued.
\begin{thm}[Differentiability of $\varphi_\gamma$]\label{thm:DerPen}
Suppose that \Cref{{Ass:fg},{Ass:fC2}} are satisfied.
Then, $\varphi_\gamma$ is continuously differentiable with
\begin{equation}\label{eq:DerPen}
\nabla \varphi_\gamma(x)=Q_\gamma(x)R_\gamma(x).
\end{equation}
If $\gamma\in (0,1/L_f)$ then the set of stationary points of $\varphi_\gamma$ equals $\zer\partial\varphi$.
\end{thm}
\begin{proof}
Consider expression \eqref{eq:EquivFBE} for $\varphi_\gamma$.
The gradient of $g^{\gamma}$ is given by \eqref{eq:nabla_e}, and since $f\in\cont^2$ we have
\begin{align*}
\nabla \varphi_\gamma(x) &= \nabla f(x) - \gamma \nabla^2 f(x)\nabla f(x) + \gamma^{-1}\left(I-\gamma\nabla^2(x)\right)(x-\gamma\nabla f(x)-T_\gamma(x)) \\
&= \left(I-\gamma\nabla^2(x)\right)(\nabla f(x) - \nabla f(x) + \gamma^{-1}(x - T_\gamma(x))).
\end{align*}
This proves \eqref{eq:DerPen}.
If $\gamma\in(0,1/L_f)$ then $Q_\gamma(x)$ is nonsingular for all \(x\), and therefore $\nabla \varphi_\gamma(x) = 0$ if and only if $R_\gamma(x) = 0$, which means that $x$ is a critical point of $\varphi$ by \eqref{eq:StatCond}.
\end{proof}

Together with \Cref{prop:EquivFBE}, \Cref{thm:DerPen} shows that if $\gamma\in (0,1/L_f)$ the nonsmooth problem~\eqref{eq:GenProb} is completely equivalent to the unconstrained minimization of the continuously differentiable function $\varphi_\gamma$, in the sense that the sets of minimizers and optimal values are equal.
In particular, as remarked in the next statement, if \(\varphi\) is convex then the set of stationary points of \(\varphi_\gamma\) turns out to be equal to the set of its minimizers, hence of solutions to the problem, even though $\varphi_\gamma$ may not be convex.
\begin{cor}\label{Cor:Convex}
Suppose that \Cref{{Ass:fg},{Ass:fC2}} are satisfied.
If $\varphi$ is convex (e.g. if $f$ is), then $\argmin\varphi=\zer\nabla\varphi_\gamma$ for all $\gamma\in (0,1/L_f)$.
\end{cor}

		\subsection{Second-order properties}
			The FBE is not everywhere twice continuously differentiable in general.
For example, if $g$ is real valued then $g^\gamma\in\cont^2$ if and only if
$g\in\cont^2$ \cite{lemarechal1997practical}.
However, second order properties will only be needed at critical points of $\varphi$ in our framework, and for this purpose  we can rely on generalized second-order differentiability notions described in~\cite[Chapter 13]{rockafellar2011variational}.
\begin{ass}\label{Ass:GenQuad}
Function \(g\) is twice epi-differentiable at $x\in\zer\partial\varphi$ for $-\nabla f(x)$, with second order epi-derivative generalized quadratic.
That is,
\begin{equation}\label{eq:GenQuadSecondEpiDer}
	\twiceepi[-\nabla f(x)]{g}{x}[d] = \innprod{d}{Md} + \delta_S(d), \quad \forall d\in\Re^n
\end{equation}
where $S\subseteq \Re^n$ is a linear subspace, and $M\in\Re^{n\times n}$ is symmetric, positive semidefinite, and such that ${\rm Im}(M)\subseteq S$ and ${\rm Ker}(M)\supseteq S^\perp$.
\end{ass}
In some results we will need to assume the following slightly stronger property.
\begin{ass}\label{Ass:Strict2Epi}
Function \(g\) satisfies \Cref{Ass:GenQuad} at \(x\in\zer\partial\varphi\) and is strictly twice epi-dif\-fer\-en\-tiable at \(x\) for \(-\nabla f(x)\).
\end{ass}
The properties of $M$ in \Cref{Ass:GenQuad} cause no loss of generality.
Indeed, letting $\Pi_S$ denote the orthogonal projector onto the linear space $S$ (which is symmetric \cite{bernstein2009matrix}), if $M\succeq 0$ satisfies \eqref{eq:GenQuadSecondEpiDer}, then so does $M'=\Pi_S[\tfrac{1}{2}(M+M^\T)]\Pi_S$, which has the wanted properties.

Twice epi-differentiability of $g$ is a mild requirement, and cases where $\twiceepi{g}{}$ is actually generalized quadratic are abundant \cite{Rockafellar1988,Rockafellar1989,Poliquin1992,poliquin1995second}.
For example, if $g$ is piecewise linear and $x\in\zer\partial\varphi$, then from \cite[Thm. 3.1]{Rockafellar1988} it follows that \eqref{eq:GenQuadSecondEpiDer} holds if and only if the normal cone $N_{\partial g(x)}(-\nabla f(x))$ is a linear subspace, which is equivalent to
\[
	-\nabla f(x)\in\relint \partial g(x)
\]
where $\relint \partial g(x)$ is the relative interior of the convex set $\partial g(x)$.
\begin{es}[Lasso] Let $A\in\Re^{m\times n}$, $b\in\Re^m$ and $\lambda>0$.
Consider $f(x)=\frac12\|Ax-b\|^2$ and $g(x) = \lambda\|x\|_1$.
Minimizing $\varphi=f+g$ is a frequent problem known as lasso, and attempts to find a sparse least squares solution to the linear system $Ax=b$.
One has
$$ \left[\partial g(x) \right]_i =
	\begin{cases}
		\set{}{\lambda} & x_i > 0 \\
		\set{}{-\lambda} & x_i < 0 \\
		[-\lambda,\lambda] & x_i = 0.
	\end{cases}
$$
In this case $\twiceepi[-\nabla f(x)]{g}{x}$ is generalized quadratic at a solution $x$ as long as whenever $x_i=0$ it holds that $|(A^T(Ax-b))_i| \neq\lambda$.
\end{es}

We begin by investigating differentiability of the residual mapping $R_\gamma$.

\begin{lem}\label{lem:DiffFPR}
Suppose that \Cref{{Ass:fg},,{Ass:fC2}} are satisfied, and that $g$ satisfies \Cref{Ass:GenQuad} (\Cref{Ass:Strict2Epi}) at a point $x\in\zer\partial\varphi$.
Then, $\prox_{\gamma g}$ is (strictly) differentiable at $x-\gamma\nabla f(x)$, and $R_\gamma$ is (strictly) differentiable at \(x\) with Jacobian
\begin{equation}\label{eq:JacFPR}
	\jac{R_\gamma}{x} = \gamma^{-1}(I-P_\gamma(x) Q_\gamma(x)),
\end{equation}
where $Q_\gamma$ is as in \eqref{Eq:Q}, and
\begin{equation}\label{eq:JacP}
	P_\gamma(x)
{}={}
	\jac{\prox_{\gamma g}}{x-\gamma\nabla f(x)}
{}={}
	\Pi_S[I+\gamma M]^{-1}\Pi_S.
\end{equation}
Moreover, $Q_\gamma(x)$ and $P_\gamma(x)$ are symmetric, $P_\gamma(x)\succeq 0$, $\|P_\gamma(x)\| \leq 1$, and if \(\gamma\in(0,1/L_f)\) then $Q_\gamma(x)\succ 0$.
\begin{proof}
See \Cref{Proof:lem:DiffFPR}.
\end{proof}
\end{lem}

Next, we see that differentiability of the residual $R_\gamma$ is equivalent to that of $\nabla\varphi_\gamma$.
Mild additional assumptions on $f$ extend this kinship to strict differentiability.
Moreover, all strong (local) minimizers of the original problem, \ie, of $\varphi$, are also strong (local)
minimizers of $\varphi_\gamma$ (and vice versa, due to the lower-bound property of $\varphi_\gamma$).

\begin{thm}\label{thm:DiffGradFBE}
Suppose that \Cref{{Ass:fg},,{Ass:fC2}} are satisfied, and that $g$ satisfies \Cref{Ass:GenQuad} at a point $x\in\zer\partial\varphi$.
Then, $\varphi_\gamma$ is twice differentiable at $x$, with symmetric Hessian given by
\begin{equation}\label{eq:HessFBE}
	\nabla^2\varphi_\gamma(x) = \gamma^{-1}Q_\gamma(x)(I-P_\gamma(x) Q_\gamma(x)),
\end{equation}
where $Q_\gamma(x)$ and $P_\gamma(x)$ are as in \Cref{lem:DiffFPR}.
If, moreover, $\nabla^2 f$ is Lipschitz continuous around $x$ and $g$ satisfies \Cref{Ass:Strict2Epi} at $x$, then $\varphi_\gamma$ is strictly twice differentiable at $x$.
\begin{proof}
Recall from \eqref{eq:DerPen} that $\nabla \varphi_\gamma(x) = Q_\gamma(x) R_\gamma(x)$.
The result follows from \Cref{lem:DiffFPR} and \Cref{prop:StrDiffProd} with $Q = Q_\gamma$ and $R = R_\gamma$.
\end{proof}
\end{thm}

\begin{thm}\label{thm:2ndOrder}
Suppose that \Cref{{Ass:fg},,{Ass:fC2}} are satisfied, and that $g$ satisfies \Cref{Ass:GenQuad} at a point $x\in\zer\partial\varphi$.
Then, for all $\gamma\in(0,1/L_f)$ the following are equivalent:
\begin{enumthm}[{label=(\alph*)},{ref=\thethm(\alph*)}]
	\item\label{it:StrLocMin2}
		$x$ is a strong local minimum for $\varphi$;
	\item\label{it:StrLocMin1}
		for all $d\in S$, $\innprod{d}{(\nabla^2 f(x) + M)d} > 0$;
	\item\label{it:StrLocMin3}
		$\jac{R_\gamma}{x}$ is similar to a symmetric and positive definite matrix;
	\item\label{it:StrLocMin4}
		$\nabla^2 \varphi_\gamma(x) \succ 0$;
	\item\label{it:StrLocMin5}
		$x$ is a strong local minimum for $\varphi_\gamma$.
\end{enumthm}
\begin{proof}
See \Cref{Proof:thm:2ndOrder}.
\end{proof}
\end{thm}

		\subsection{Interpretations}
			An interesting observation is that the FBE provides a link between gradient methods and FBS, just like the Moreau envelope~\eqref{eq:MoreauEnv}
does for the proximal point algorithm~\cite{rockafellar1976monotone}. To see this, consider the problem
\begin{equation}\label{eq:NSprob}
\minimize\ g(x)
\end{equation}
where $g\in\Gamma_0(\OurSpace)$. The proximal point algorithm for
solving~\eqref{eq:NSprob} is
\begin{equation}\label{eq:ProxMin}
x^{k+1}=\prox_{\gamma g}(x^k).
\end{equation}
It is well known that the proximal point algorithm can be interpreted as a gradient method for minimizing the Moreau envelope of $g$, cf.~\eqref{eq:MoreauEnv}.
Indeed, due to~\eqref{eq:nabla_e}, iteration~\eqref{eq:ProxMin} can be expressed as
$$x^{k+1}=x^k-\gamma\nabla g^\gamma(x^k).$$
This simple idea provides a link between nonsmooth and smooth optimization and has led to the discovery of a variety of algorithms for problem~\eqref{eq:NSprob},
such as semismooth Newton methods~\cite{fukushima1996globally}, variable-metric~\cite{bonnans1995family} and quasi-Newton methods~\cite{mifflin1998quasi,chen1999proximal,burke2000superlinear},
and trust-region methods~\cite{sagara2005trust}, to name a few.

However, when dealing with composite problems, even if $\prox_{\gamma f}$ and $\prox_{\gamma g}$
are cheaply computable, computing the proximal mapping of $\varphi = f+g$ is usually as hard as
solving~\eqref{eq:GenProb} itself.
On the other hand, forward-backward splitting takes advantage of the structure of the
problem by operating separately on the two summands, cf.~\eqref{eq:FBS}.
The question that naturally arises is the following:
\begin{quote}
\emph{Is there a continuously differentiable function that provides an
interpretation of FBS as a gradient method, just like the Moreau envelope does
for the proximal point algorithm?} 
\end{quote}
The forward-backward envelope provides an affirmative answer. Specifically, whenever $f$ is $\cont^2$,
FBS can be interpreted as the following (variable-metric) gradient method on the FBE:
\begin{equation}x^{k+1}=x^k-\gamma(I-\gamma\nabla^2 f(x^k))^{-1}\nabla \varphi_\gamma(x^k),\label{eq:VariableMetric}\end{equation}
cf. \Cref{thm:DerPen}. Furthermore, the following properties hold for the Moreau envelope
\begin{equation*}
g^{\gamma} \leq g,\quad\inf g^{\gamma} = \inf g,\quad\argmin g^{\gamma} = \argmin g,
\end{equation*}
which correspond to \Cref{{prop:UppBnd},,{prop:EquivFBE}} for the FBE.
The relationship between Moreau envelope and forward-backward envelope is then
apparent. This opens the possibility of extending FBS and devising new algorithms for problem~\eqref{eq:GenProb} by simply reconsidering
and appropriately adjusting methods for unconstrained minimization of continuously differentiable functions, the most well studied problem in
optimization.

	\section{Forward-backward line-search methods}
		\label{SEC:Algorithms}
		We consider line-search methods applied to the problem of minimizing $\varphi_\gamma$, hence solving \eqref{eq:GenProb}.
Requirements of such methods are often restrictive, including convexity or even strong convexity of the objective function, properties that unfortunately the FBE does not satisfy in general.
As opposed to this, FBS possesses strong convergence properties and complexity estimates.
We now show that it is possible to exploit the composite structure of~\eqref{eq:GenProb}
and devise line-search methods with the same global convergence properties and oracle information as FBS.

\begin{algorithm}%
	\caption{Forward-backward line-search method with adaptive $\gamma$}%
	\label{alg:Global}%
	\begin{algorithmic}[1]
  \Require $x^0\in \OurSpace$, $\gamma_0 > 0$, $\sigma\in(0,1)$, $\beta\in[0,1)$, $k\gets 0$
  \If{$R_{\gamma_k}(x^k)=0$} stop \label{step:GlobalStopping}
  \EndIf
  \State select $d^k$ such that $\innprod{d^k}{\nabla\varphi_{\gamma_k}(x^k)}\leq 0$ \label{step:GlobalDirection}
  \State select $\tau_k \geq 0$ and set $w^k\gets x^k+\tau_k d^k$ such that $\varphi_{\gamma_k}(w^k)\leq \varphi_{\gamma_k}(x^k)$ \label{step:GlobalStepsize}
  \If{$\varphi(T_{\gamma_k}(w^k)) +\frac{\beta\gamma_k}{2}\|R_{\gamma_k}(w^k)\|^2> \varphi_{\gamma_k}(w^k)$} \label{step:GlobalCheck} $\gamma_k\gets \sigma \gamma_k$, go to step \ref{step:GlobalFixedStopping}
  \Else\ $\gamma_{k+1}\gets \gamma_k$
  \EndIf
  \State $x^{k+1}\gets T_{\gamma_k}(w^k)$ \label{step:GlobalFB}
  \State $k \gets k+1$, go to step \ref{step:GlobalFixedStopping}
\end{algorithmic}

\end{algorithm}

\Cref{alg:Global} interleaves descent steps over the FBE with forward-backward steps. 
In particular, steps~\ref{step:GlobalDirection} and~\ref{step:GlobalStepsize} provide fast asymptotic convergence when
directions $d^k$ are appropriately selected, while step~\ref{step:GlobalFB} ensures global convergence: this is of central importance, as such properties are not usually enjoyed by standard line-search methods employed to minimize general nonconvex functions
\cite{Dai2002,Mascarenhas2004,Mascarenhas2007,Dai2013}.
Moreover, in the convex case we are able to show global convergence rate results which are not typical for line-search methods with, e.g., quasi-Newton directions.
We anticipate some of the favorable properties that \Cref{alg:Global} shares with FBS:
\begin{itemize}
	\item square-summability of the residuals for lower bounded \(\varphi\) (\Cref{prop:GlobalDesc});
	\item global convergence when \(\varphi\) has bounded level sets and satisfies the Kurdyka-{\L}ojasiewicz at its stationary points
	(\Cref{thm:Convergence});
	\item global sublinear rate of the objective for convex \(\varphi\) with bounded level sets (\Cref{th:RateGlobalConv});
	\item local linear rate when $\varphi$ has the {\L}ojasiewicz property at its critical points (\Cref{thm:LocalLinConvKL}).
\end{itemize}
Moreover, unlike ordinary line-search methods applied to \(\varphi_\gamma\), we will see in \Cref{prop:GlobalDesc}
that \Cref{alg:Global} is a descent method both for both \(\varphi_\gamma\) and \(\varphi\).
Note that, despite the fact that the algorithm operates on \(\varphi_\gamma\), all the above properties require assumptions or
provide results on \(\varphi\), \ie, on the original problem.

The parameter $\gamma$ defining the FBE is adjusted in step~\ref{step:GlobalCheck}
so as to comply with the inequality in \Cref{prop:LowBnd},
starting from an initial value $\gamma_0$ and decreasing it when necessary.
The next result shows that $\gamma_0$ is decremented only a finite number of times along the iterations, and therefore $\gamma_k$ is positive and eventually constant.
In the rest of the paper we will denote \(\gamma_\infty\) such asymptotic value of \(\gamma_k\).

\begin{lem}\label{lem:LowerBoundGamma}
Let $\seq{\gamma_k}$ the sequence of stepsize parameters computed by \Cref{alg:Global}, and let 
\(
	\gamma_\infty
{}={}
	\min_{i\in\Nn} \gamma_i
\).
Then for all $k\in\Nn$,
\[
	\gamma_k
{}\geq{}
	\gamma_\infty
{}\geq{}
	\min\set{}{\gamma_0, \sigma(1-\beta)/L_f}
{}>{}
	0.
\]
\begin{proof}
See \Cref{Proof:lem:LowerBoundGamma}.
\end{proof}
\end{lem}

\begin{rem}\label{rem:Algorithm1}
In \Cref{alg:Global}:
\begin{enumrem}
	\item Selecting $\beta = 0$ and $d^k\equiv 0$, $\tau_k\equiv 0$ for all $k$ yields the classical forward-backward
		splitting with backtracking on $\gamma$ \cite[Sec.~3]{beck2009fast}.

	\item\label{it:LineSearch} Substituting step~\ref{step:GlobalFB} with $x^{k+1}\gets w^k$ yields a classical line-search method for the problem of minimizing $\varphi_\gamma$, where a suitable $\gamma$ is adaptively determined.
		However, extensive numerical experience has shown that even though this variant seems to always converge, our choice $x^{k+1}\gets T_{\gamma_k}(w^k)$ usually performs better in practice, in terms of number of forward-backward steps, cf. \Cref{SEC:Simulations}.

	\item Step~\ref{step:GlobalFB} 
	  comes at no additional cost once $\tau_k$
		has been determined by means of a line-search. In fact, in order to evaluate $\varphi_{\gamma_k}(w^k)$
		and test the condition in step~\ref{step:GlobalStepsize}, the evaluation of $T_{\gamma_k}(w^k)$ is
		required.

	\item When $L_f$ is known and $\gamma_0 \in (0,(1-\beta)/L_f]$, the condition in step~\ref{step:GlobalCheck}
	  never holds, see \Cref{prop:LowBnd}.
	  In this case \Cref{alg:Global} reduces to \Cref{alg:GlobalFixed}: without loss of generality
	  we will analyze the first.
\end{enumrem}
\end{rem}

\begin{algorithm}%
	\caption{Forward-backward line-search method with fixed $\gamma$}%
	\label{alg:GlobalFixed}%
	\begin{algorithmic}[1]
  \Require $x^0\in \OurSpace$, $\beta\in[0,1)$, $\gamma \in (0,(1-\beta)/L_f)$, $k\gets 0$
  \If{$R_{\gamma_k}(x^k)=0$} stop \label{step:GlobalFixedStopping}
  \EndIf
  \State select $d^k$ such that $\innprod{d^k}{\nabla\varphi_{\gamma}(x^k)}\leq 0$ \label{step:GlobalFixedDirection}
  \State select $\tau_k \geq 0$ and set $w^k\gets x^k+\tau_k d^k$ such that $\varphi_{\gamma}(w^k)\leq \varphi_{\gamma}(x^k)$ \label{step:GlobalFixedStepsize}
  \State $x^{k+1}\gets T_{\gamma}(w^k)$ \label{step:GlobalFixedFB}
  \State $k \gets k+1$, go to step \ref{step:GlobalFixedStopping}
\end{algorithmic}

\end{algorithm}


We denote by  $\omega(x^0)$ the set of cluster points of the sequence $\seq{x^k}$ produced
by \Cref{alg:Global} started from $x^0\in\Re^n$. The following result states that \Cref{alg:Global}
is a descent method both for the FBE \(\varphi_\gamma\) and for the original function \(\varphi\), and, as it holds for FBS, that the sequence of fixed-point residuals is square-summable if the function is lower bounded.

\begin{prop}[Subsequential convergence]\label{prop:GlobalDesc}
Suppose that \Cref{Ass:fg} is satisfied.
Then, the following hold for the sequences generated by \Cref{alg:Global}:
\begin{enumprop}
	\item\label{prop:FunDecr}
		\(\displaystyle
			\varphi(x^{k+1}) \leq
			\varphi(x^k)-\tfrac{\beta\gamma_k}{2}\|R_{\gamma_k}(w^k)\|^2-\tfrac{{\gamma_k}}{2}\|R_{\gamma_k}(x^k)\|^2
		\)~
		for all \(k\in\Nn\);
	\item\label{prop:Rxl2}
		either $\seq{\|R_{\gamma_k}(x^k)\|}$ is square summable, or \(\varphi(x^k)\to\inf\varphi=-\infty\), in which case $\omega(x^0)=\emptyset$;
	\item\label{prop:ClusterCritical}
		$\omega(x^0)\subseteq \zer\partial\varphi$, \ie, every cluster point of $\seq{x^k}$ is critical;
	\item\label{prop:Rwl2}
		if $\beta>0$, then either $\seq{\|R_{\gamma_k}(w^k)\|}$ is square summable and every cluster point of $\seq{w^k}$ is critical, or \(\varphi_{\gamma_k}(w^k)\to\inf\varphi=-\infty\) in which case \(\seq{w^k}\) has no cluster points.
\end{enumprop}
\end{prop}
\begin{proof}
See \Cref{Proof:prop:GlobalDesc}.
\end{proof}

An immediate consequence is the following result concerning the convergence
of the fixed-point residual.

\begin{thm}\label{th:RateGlobal}
Suppose that \Cref{Ass:fg} is satisfied, and consider the sequences generated by \Cref{alg:Global}.
Then,
\begin{align*}
	\min_{i=0\cdots k}{
		\|R_{\gamma_{i}}(x^i)\|^2
	}
{}\leq{} &
	\frac{2}{(k+1)}
	\frac{
		\varphi(x^0)-\inf\varphi
	}{
		\min\set{}{\gamma_0, \sigma(1-\beta)/L_f}
	}.
\intertext{If $\beta > 0$, then for all $k\in\Nn$ we also have}
	\min_{i=0\cdots k}{
		\|R_{\gamma_{i}}(w^i)\|^2
	}
{}\leq{} &
	\frac{2}{(k+1)}
	\frac{
		\varphi(x^0)-\inf\varphi
	}{
		\beta\min\set{}{\gamma_0, \sigma(1-\beta)/L_f}
	}.
\end{align*}
\end{thm}
\begin{proof}
See \Cref{Proof:th:RateGlobal}.
\end{proof}

We now analyze the convergence properties of \Cref{alg:Global}.
We first consider the case where $f$ is convex, and later discuss the general case under the assumption that $\varphi$ has the Kurdyka-{\L}ojasiewicz property.
In this distinction we also highlight the favorable property that our method enjoys in the convex case, namely the total freedom for selecting the directions \(d^k\) without any assumptions.
When the directions are selected, say, according to a quasi-Newton scheme \(d^k=-B_k^{-1}\nabla\varphi_\gamma(x^k)\) where \(B_k\) approximates in some sense the Hessian of \(\varphi_\gamma\) at \(x^k\), linear convergence to a strong minimum and sublinear convergence of the objective in general are not affected if \(\seq{B_k^{-1}}\) turned out to be unbounded.
In the nonconvex case, instead, boundedness of \(\seq{B_k^{-1}}\) will be necessary for the sake of global convergence when the Kurdyka-\L ojasiewicz property holds for \(\varphi\).
The latter is however a milder assumption with respect to usual nonconvex line-search methods where \(\seq{B_k^{-1}}\) is required to have bounded condition number or \(\seq{d^k}\) to be \emph{gradient-oriented} (see \cite{noll2013convergence} and the references therein).

		\subsection{Convergence in the convex case}
			\label{SUBSEC:ConvexCase}
			We now prove that when $f$ is convex \Cref{alg:Global} converges to the optimal objective value with the same sublinear rate as FBS.
Notice that we require convexity of the original function $f$ (and $g$), and \emph{not} that of \(\varphi_\gamma\) which may
fail to be convex even when \(\varphi\) is.

\begin{thm}[Global sublinear convergence]\label{th:RateGlobalConv}
Suppose that \Cref{{Ass:fg},{Ass:LevelSets}} are satisfied, and that \(f\) is convex.
Then, for the sequences generated by \Cref{alg:Global}, either $\varphi(x^0)-\inf\varphi \geq R^2/{\gamma_0}$ and
\begin{equation}\label{eq:FirstStep}
\varphi(x^1)-\inf\varphi \leq\frac{R^2}{2{\gamma_0}},
\end{equation}
or for any $ k\in\Nn$ it holds
\begin{equation}\label{eq:kStep}
\varphi(x^{ k})-\inf\varphi \leq\frac{2R^2}{k\min\set{}{\gamma_0, \sigma(1-\beta)/L_f}}.
\end{equation}
\end{thm}
\begin{proof}
See \Cref{Proof:th:RateGlobalConv}.
\end{proof}

In the following result we see that the convergence rate of \(\seq{x^k}\) is linear when close to a strong local minimum.
\begin{thm}[Local linear convergence]\label{thm:LocalLinConvCVX}
Suppose that \Cref{Ass:fg} is satisfied.
Suppose further that \(f\) is convex and that $x_\star$ is a strong (global) minimum of $\varphi$, \ie, there exist a neighborhood $N$ of $x_\star$ and $c>0$ such that
\begin{equation}\label{eq:SOSCphi}
	\varphi(x) - \varphi(x_\star) \geq \tfrac{c}{2}\|x-x_\star\|^2,\quad \forall x\in N.
\end{equation}
Then, there is $k_0\geq 0$ such that the subsequences \(\seq{\varphi(x^k)}[k\geq k_0]\) and \(\seq{\varphi_{\gamma_k}(w^k)}[k\geq k_0]\)
produced by \Cref{alg:Global} converge $Q$-linearly to $\varphi(x_\star)$ with factor $\omega$,%
\footnote{We say that \(\seq{x^k}\) converges linearly to \(x_\star\) with factor \(\omega\) if \(\|x^{k+1}-x_\star\|\leq\omega\|x^k-x_\star\|\) for all \(k\).}
and $\seq{x^k}[k\geq k_0]$ converges $Q$-linearly to $x_\star$ with factor $\sqrt{\omega}$,
where
\[
	\omega
{}\leq{}
	\max\set{}{
		\tfrac12,
		1-\tfrac c4 \min\set{}{
			\gamma_0,\sigma(1-\beta)/(4L_f)
		}
	}
{}\in{}
	\bigl[1/2, 1\bigr).
\]
Moreover, if $x_\star$ is a strong (global) minimum for $\varphi_{\gamma_\infty}$, where \(\gamma_\infty\) is as in \Cref{lem:LowerBoundGamma}, then also $w^k \to x_\star$ with local linear rate.
\begin{proof}
See \Cref{Proof:thm:LocalLinConvCVX}.
\end{proof}
\end{thm}

The introduction of \(\gamma_\infty\) in the statement above is due to the fact that \(\gamma_k\) may vary over the iterations.
However, under the assumptions of \Cref{thm:2ndOrder} and if $\gamma_\infty < 1/L_f$, then the requirement of \(x_\star\) to be a strong local minimizer for \(\varphi_{\gamma_\infty}\) is superfluous, as it is already implied by strong local minimimality of $x_\star$ for $\varphi$.

\begin{cor}[Global linear convergence]\label{cor:GlobalLinConvCVX}
Suppose that \Cref{Ass:fg} is satisfied, that \(f\) is convex and that \(\varphi\) is strongly convex (e.g. if \(f\) is strongly convex).
Then, the sequences $\seq{x^k}$, $\seq{\varphi(x^k)}$ and $\seq{\varphi_{\gamma_k}(w^k)}$ generated by \Cref{alg:Global}
converge with global linear rate.
\begin{proof}
In this case \Cref{thm:LocalLinConvCVX} applies with $N = \Re^n$, $c = \mu_\varphi$ (the convexity modulus of $\varphi$)
and $k_0 = 0$.
\end{proof}
\end{cor}

		\subsection{Convergence under KL assumption}
			\label{SUBSEC:KL}
			We now analyze the convergence of the iterates of \Cref{alg:Global} to a critical point
under the assumption that $\varphi$ satisfies the Kurdyka-{\L}ojasiewicz (KL) property
\cite{lojasiewicz1963propriete,lojasiewicz1993geometrie,kurdyka1998gradients}.
For related works exploiting this property in proving convergence of optimization
algorithms such as FBS we refer the reader to \cite{Attouch2009,Attouch2010,Attouch2013,bolte2014proximal,Ochs2014}.

\begin{defin}[KL property {\cite[Def. 3]{bolte2014proximal}}]\label{def:KLproperty}
A proper lower semi-continuous function $\varphi:\Re^n\to\Rinf$ has the Kurdyka-{\L}ojasiewicz property (KL) at $x_\star\in\dom\partial\varphi$ if there exist $\eta\in (0,+\infty]$, a neighborhood $U$ of $x_\star$, and a continuous concave function $\psi:[0,\eta]\to[0,+\infty)$ such that:
\begin{enumdefin}
	\item\label{Def:KL1} $\psi(0)=0$, 
	\item\label{Def:KL2} $\psi$ is $\cont^1$ on $(0,\eta)$, 
	\item\label{Def:KL3} $\psi'(s)>0$ for all  $s\in (0,\eta)$,
	\item\label{it:KLinequality} for every $x\in U\cap\set{x\in\Re^n}{\varphi(x_\star)<\varphi(x)\leq\varphi(x_\star)+\eta}$,
	$$\psi'(\varphi(x)-\varphi(x_\star))\dist(0,\partial\varphi(x))\geq 1.$$
\end{enumdefin}
We say that \(\varphi\) has the KL property on \(S\subseteq\Re^n\) it has the KL property on every \(x\in S\).
\end{defin}

Function $\psi$ in the previous definition is usually called \emph{desingularizing function}.
All subanalytic functions which are continuous over their domain have the KL property \cite{Bolte2007}.
Under the KL assumption we are able to prove the following convergence result.
Once again, we remark that such property is required on the original function \(\varphi\), rather than on the surrogate \(\varphi_\gamma\).

\begin{thm}\label{thm:Convergence}
Suppose that \Cref{{Ass:fg},{Ass:LevelSets}} are satisfied, and that $\varphi$ satisfies the KL property on \(\omega(x^0)\) (e.g., if it has it on \(\zer\partial\varphi\)).
Suppose further that in \Cref{alg:Global} $\beta>0$, and that there exist $\bar{\tau},c>0$ such that
$\tau_k\leq\bar{\tau}$ and $\|d^k\|\leq c\|R_{\gamma_k}(x^k)\|$ for all \(k\in\Nn\).
Then, the sequence of iterates $\seq{x^k}$ is either finite and ends with $R_{\gamma_k}(x^k)=0$, or converges to a critical point $x_\star$ of $\varphi$.
\end{thm}
\begin{proof}
See \Cref{Proof:thm:Convergence}.
\end{proof}

In case where $\varphi$ is subanalytic, the desingularizing function can be taken of
the form $\psi(s) = \sigma s^{1-\theta}$, for $\sigma > 0$
and $\theta \in [0, 1)$ \cite{Bolte2007}. In this case, the condition in \Cref{it:KLinequality}
is referred to as {\L}ojasiewicz inequality. Depending on the value of $\theta$ we can
derive local convergence rates for \Cref{alg:Global}.

\begin{thm}[Local linear convergence]\label{thm:LocalLinConvKL}
Suppose that \Cref{{Ass:fg},{Ass:LevelSets}} are satisfied, and that $\varphi$ satisfies the KL property on \(\omega(x^0)\) (e.g., if it has it on \(\zer\partial\varphi\)) with
\begin{equation}\label{eq:LojasiewiczProperty}
	\psi(s) = \sigma s^{1-\theta}
\qquad\text{for some}~~
	\sigma>0
	~~\text{and}~~
	\theta \in (0, \tfrac12].
\end{equation}
Suppose further that in \Cref{alg:Global} $\beta>0$, and that there exist $\bar{\tau},c>0$ such that
$\tau_k\leq\bar{\tau}$ and $\|d^k\|\leq c\|R_{\gamma_k}(x^k)\|$ for all \(k\in\Nn\).
Then, the sequence of iterates $\seq{x^k}$ converges to a point $x_\star\in\zer\partial\varphi$ with $R$-linear rate.
\begin{proof}
See \Cref{Proof:thm:LocalLinConvKL}.
\end{proof}
\end{thm}

	\section{Quasi-Newton methods}
		\label{SEC:QuasiNewton}
		We now turn our attention to choices of the direction $d^k$ in \Cref{alg:Global}.
Applying classical quasi-Newton methods \cite{dennis1977quasi} to the problem of minimizing $\varphi_\gamma$ yields,
starting from a given $x^0$,
\begin{align*}
d^k &= -B_k^{-1}\nabla\varphi_\gamma(x^k),\\
x^{k+1} &= x^k+\tau_k d^k,
\end{align*}
where $B_k$ is nonsingular and chosen so as to approximate (in some sense) the Hessian of $\varphi_\gamma$ at $x^k$,
and stepsize $\tau_k>0$ is selected with a line-search procedure enforcing a sufficient decrease condition.
However, the convergence properties of
quasi-Newton methods are quite restrictive. The BFGS algorithm
is guaranteed to converge, in the sense that
$$\liminf_{k\to\infty}\|\nabla\varphi_\gamma(x^k)\|=0,$$
when the objective is convex \cite{powell1976some}.
Its limited memory variant, L-BFGS, requires strong convexity to guarantee convergence,
and in that case the cost is shown to converge $R$-linearly to the optimal value \cite{liu1989largescale}.
Moreover, there exist examples of nonconvex function for which
quasi-Newton methods need not converge to critical points
\cite{Dai2002,Mascarenhas2004,Mascarenhas2007,Dai2013}.
To overcome this, we consider quasi-Newton directions in the setting of \Cref{alg:Global}.

The resulting methods enjoy the same global convergence properties illustrated in \Cref{SEC:Algorithms} and superlinear asymptotic convergence under standard assumptions:
we will assume, as it is usual, (strict) differentiability of $\nabla\varphi_\gamma$ and nonsingularity of $\nabla^2\varphi_\gamma$ at a critical point.
Properties of $f$ and $g$ that guarantee these requirements were discussed in \Cref{thm:2ndOrder,thm:DiffGradFBE}: if \(\gamma=\gamma_\infty\) is as in \Cref{lem:LowerBoundGamma}, then (strict) differentiability of \(\nabla\varphi_\gamma\) at \(x_\star\in\zer\partial\varphi\) and positive definiteness of \(\nabla^2\varphi_\gamma(x_\star)\) are ensured if \Cref{Ass:GenQuad} (\Cref{Ass:Strict2Epi}) holds, \(x_\star\) is a strong local minimum for \(\varphi\), and \(\gamma<1/L_f\).

The following result gives for the proposed algorithmic scheme the analogous of the Dennis-Mor\'e condition,
see \cite[Thm. 2.2]{dennis1974characterization} and \cite[Thm. 3.3]{ip1992local}.
Differently from the cited results, we fit the analysis to our algorithmic framework where an additional forward-backward step is operated.
Furthermore, in \Cref{thm:SuperlinearConvergence2} we will see how achieving superlinear convergence is possible without the need to ensure \emph{sufficient} decrease in the objective, or even to consider direction of strict descent, but simply with the nonincrease conditions of steps~\ref{step:GlobalDirection} and~\ref{step:GlobalStepsize}.
This contrasts with the usual requirements of classical line-search methods,
where instead a sufficient decrease must be enforced in order for the sequence of iterates to converge.
In \Cref{alg:Global}, in fact, such decrease is guaranteed by the final update in step \ref{step:GlobalFB}.

\begin{thm}\label{thm:SuperlinearConvergence1}
Suppose that \Cref{Ass:fg} is satisfied, and let $\gamma>0$.
Suppose that \(\nabla\varphi_\gamma\) is strictly differentiable at \(x_\star\), and that \(\nabla^2\varphi_\gamma(x_\star)\) is nonsingular.
Let $\seq{B_k}$ be a sequence of nonsingular \(\Re^{n\times n}\)-matrices and suppose that for some $x^0\in\Re^n$ the sequences $\seq{x^k}$ and $\seq{w^k}$ generated by
\begin{align*}
	w^k
{}={}
	x^k - B_k^{-1}\nabla\varphi_\gamma(x^k)
\quad\text{and}\quad
	x^{k+1}
{}={}
	T_\gamma(w^k)
\end{align*}
converge to $x_\star$.
If $x^k, w^k\notin \zer\partial\varphi$ for all $k\geq 0$ and
\begin{equation}\label{eq:DennisMoreCondition}
	\lim_{k\to\infty}{
		\frac{
			\|(B_k - \nabla^2\varphi_\gamma(x_\star))(w^k-x^k)\|
		}{
			\|w^k-x^k\|
		}
	}
{}={}
	0,
\end{equation}
then $x^k\to x_\star$ $Q$-superlinearly, and $w^k\to x_\star$ $R$-superlinearly.
\end{thm}
\begin{proof}
See \Cref{Proof:thm:SuperlinearConvergence1}.
\end{proof}

To obtain superlinear convergence of \Cref{alg:Global} when quasi-Newton directions are used and condition \eqref{eq:DennisMoreCondition} on the sequence $\seq{B_k}$ holds, we must verify that eventually $\varphi_\gamma(x^k+d^k) \leq \varphi_\gamma(x^k)$, so that the stepsize $\tau_k = 1$ is accepted in step \ref{step:GlobalStepsize} and the iterations reduce to those described in \Cref{thm:SuperlinearConvergence1}.

\begin{thm}\label{thm:SuperlinearConvergence2}
Suppose that \Cref{Ass:fg} is satisfied, and that in \Cref{alg:Global} direction $d^k$ is set as
\[
	d^k = -B_k^{-1}\nabla\varphi_{\gamma_k}(x^k)
\]
for a sequence of nonsingular matrices $\seq{B_k}$ satisfying~\eqref{eq:DennisMoreCondition}, with $\tau_k = 1$ being tried first in step \ref{step:GlobalStepsize}.
Let \(\gamma=\gamma_\infty\) as in \Cref{lem:LowerBoundGamma}, and suppose further that the sequences $\seq{x^k}$ and $\seq{w^k}$
converge to a critical point $x_\star$ at which $\nabla\varphi_\gamma$ is continuously semidifferentiable with $\nabla^2\varphi_\gamma(x_\star)\succ 0$.
Then, $x^k\to x_\star$ $Q$-superlinearly and $w^k\to x_\star$ $R$-superlinearly.
\end{thm}
\begin{proof}
See \Cref{Proof:thm:SuperlinearConvergence2}.
\end{proof}

		\subsection{BFGS}
The sequence $\seq{B_k}$ can be computed using BFGS updates: starting from $B_0\succ 0$, use vectors
\begin{subequations}\label{Eq:BFGS}
\begin{equation}
s^{k} = w^k-x^k, \quad y^{k} = \nabla\varphi_{\gamma}(w^k)-\nabla\varphi_{\gamma}(x^k),\label{eq:BFGSpairs}\end{equation}
to compute
\begin{equation}
B_{k+1} = 
	\begin{cases}
		B_{k} + \frac{y^{k}(y^{k})^\T}{\innprod{y^{k}}{s^{k}}} - \frac{B_{k}s^{k}(B_{k}s^{k})^\T}{\innprod{s^{k}}{B_{k}s^{k}}} &
			\mbox{if } \innprod{s^k}{y^k}>0, \\ 
		B_{k} & \mbox{otherwise.}
	\end{cases}
\label{eq:BFGSupdate}
\end{equation}
\end{subequations}
Note that in this way $B_k\succ 0$, for all $k\geq 0$, and
$d^k = -B^{-1}\nabla\varphi_{\gamma}(x^k)$ is always a direction of descent for $\varphi_\gamma$.
No matrix inversion is needed to compute $d^k$ in practice, since it is possible to perform the
inverse updates of~\eqref{eq:BFGSupdate} directly producing the sequence $\seq{B_k^{-1}}$, see \cite{dennis1977quasi,wright1999numerical}.

In light of the convergence results for \Cref{alg:Global} given in \Cref{SEC:Algorithms}
we now proceed under either of the following assumptions.

\begin{ass}
Function $\varphi$
\begin{enumass}
	\item\label{ass:Convex}
		is either convex and such that
		$\varphi(x) - \varphi(x_\star) \geq \tfrac{c}{2}\|x-x_\star\|^2$,
		for some $c>0$ and all $x$ close enough to $x_\star$, the unique minimizer of $\varphi$;
	\item\label{ass:KL}
		or it has the KL property on \(\omega(x^0)\) with
		$\psi(s) = \sigma s^{1-\theta}$,
		where $\sigma>0$ and $\theta \in (0, \tfrac{1}{2}]$.
\end{enumass}
\end{ass}

\begin{thm}\label{thm:SuperlinearConvergenceBFGS}
Suppose that \Cref{Ass:fg} is satisfied, and that in \Cref{alg:Global} directions $d^k$ are set as
\[
	d^k = -B_k^{-1}\nabla\varphi_{\gamma_k}(x^k)
\qquad
	\text{with $B_k$ as in \eqref{Eq:BFGS},}
\]
and with $\tau_k = 1$ being tried first in step \ref{step:GlobalStepsize}.
Let \(\gamma=\gamma_\infty\) as in \Cref{lem:LowerBoundGamma}, and suppose further that the sequences $\seq{x^k}$ and $\seq{w^k}$
converge to a critical point $x_\star$ at which $\nabla\varphi_\gamma$ is calmly semidifferentiable (see \Cref{prop:Calm}) with $\nabla^2\varphi_\gamma(x_\star)\succ 0$.
Then, $x^k\to x_\star$ $Q$-superlinearly and $w^k\to x_\star$ $R$-superlinearly.
\end{thm}

\begin{proof}
See \Cref{Proof:thm:SuperlinearConvergenceBFGS}.
\end{proof}


		\subsection{L-BFGS}
			When dealing with a large number of variables, storing (and updating) approximations of the
Hessian matrix (or its inverse) may be impractical.
Limited-memory quasi-Newton methods remedy this by storing, instead of a dense $n\times n$
matrix, only a few most recent pairs $(s^k,y^k)$ implicitly representing such approximation.
The lim\-it\-ed-memory BFGS method (L-BFGS) is probably the most widely used method of this class, and
was first introduced in \cite{liu1989largescale}. It is based on the BFGS update, but uses at iteration
$k$ only the most recent $\tilde{m}=\min\set{}{m,k}$ pairs (here $m$ is a parameter, usually
$m\in\set{}{3,\ldots,20}$) to compute a descent direction: $d^k$ is obtained using a procedure
known as \emph{two-loop recursion}~\cite{nocedal1980updating}, so that no matrix storage is required,
and in fact only $O(n)$ operations are needed. For this reason, limited-memory methods like L-BFGS
are much more suited for large scale applications.
%
%
Similarly to BFGS,
a safeguard is used to make sure that $\innprod{s^k}{y^k}>0$, so that
$d^k$ is always a descent direction for $\varphi_{\gamma_k}$.

	\section{Simulations}
		\label{SEC:Simulations}
		We now present numerical results obtained with the proposed methods.
We compare \Cref{{alg:Global},{alg:GlobalFixed}}, using BFGS and L-BFGS directions, against the corresponding classical
line-search method (see \Cref{it:LineSearch}), forward-backward splitting (FBS) and its
accelerated variant (Fast FBS) in the case of convex problems.
We set $\beta = 0.05$ in \Cref{alg:Global}, therefore when $L_f$ is known and \Cref{alg:GlobalFixed} is applied instead, then
we set $\gamma = 0.95/L_f$.
To determine the stepsize $\tau_k$ in \Cref{alg:Global} we use a simple backtracking line-search verifying the condition
$\varphi_{\gamma_k}(x^k+\tau_k d^k)\leq \varphi_{\gamma_k}(x^k)$. In classical BFGS and L-BFGS we use instead a procedure
enforcing the usual Wolfe conditions, inspired by \cite[Sec.~II.3.3]{hiriart1996convex}: although simpler, this strategy
performed favorably with respect to other algorithms in our tests, see \cite[Sec.~1.2]{bertsekas1999nonlinear}, \cite[Sec.~3]{wright1999numerical},
\cite[Sec.~2.6]{fletcher2013practical}.
We always set the memory parameter $m=5$ when computing L-BFGS directions.

All experiments were performed in MATLAB. The implementation of the methods used in the
tests are made available by the authors.\footnote{\url{http://github.com/kul-forbes/ForBES}}

		\subsection{Lasso}
			The problem is to find a sparse representation of a vector $b\in\Re^m$ as combination
of the columns of $A\in\Re^{m\times n}$. This is done by minimizing $\varphi = f+g$ where
$$ f(x) = \tfrac{1}{2}\|Ax-b\|_2^2,\quad g(x) = \lambda\|x\|_1. $$
The proximal mapping of $g$ is the soft-thresholding operation,
while the computationally relevant operation here is the computation of $f$ and $\nabla f$, which
involves matrix-vector products with $A$ and $A^\T$.
Parameter $\lambda$ modulates between a small residual
$Ax_\star-b$ and a sparse solution vector $x_\star$, \ie, the larger $\lambda$ the
more zero coefficients $x_\star$ has. In particular, $\lambda_\max = \|A^\T b\|_\infty$
is the minimum value such that for $\lambda \geq \lambda_\max$ the solution is $x_\star = 0$.
We have $L_f = \|A^\T A\|$, which can be quickly approximated for using power iteration.


We ran the algorithms on the SPEAR datasets.\footnote{\url{http://wwwopt.mathematik.tu-darmstadt.de/spear/}}
Among these, we considered the 274 \emph{high dynamic range} datasets, \ie,
the ones for which the ratio $\max|x_\star|_i / \min|x_\star|_i$ is large, and
are usually the hardest instances for algorithms to solve.
First we considered BFGS directions in \Cref{alg:GlobalFixed}, on a small dimensional instance and the results are
illustrated in \Cref{fig:bpdn_small}. We computed $x_\star$ and $\varphi_\star$ (approximately) by running
one of the methods with very restrictive termination criteria.
Then we considered the largest instances, using L-BFGS directions in \Cref{alg:GlobalFixed}.
Performance of the algorithms is detailed in \Cref{tab:BPDN1}.
We stopped the algorithms as soon as $$ \varphi(x^k)-\varphi_\star \leq 10^{-6}(1+|\varphi_\star|). $$
It is apparent that \Cref{alg:GlobalFixed} requires much less iterations and operations than fast FBS.

\begin{figure}[tbp]
\centering
	\includegraphics{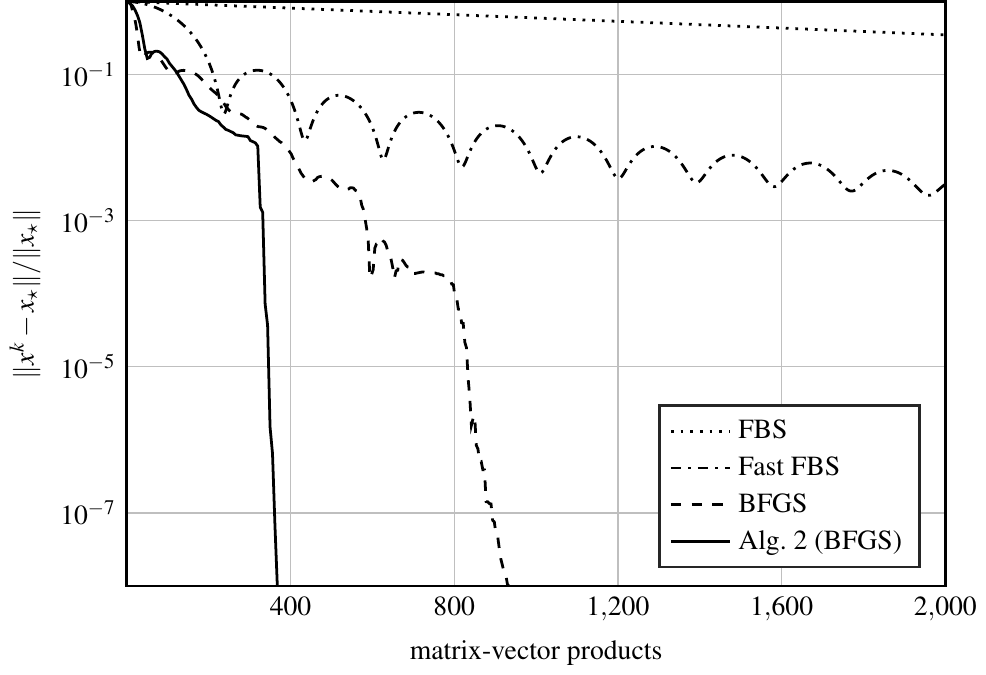}
	\caption{Lasso: performance comparison between forward-back\-ward splitting, the standard BFGS algorithm
		applied to the minimization of $\varphi_\gamma$, and the proposed \Cref{alg:GlobalFixed} (with BFGS directions).
		The dataset used is \texttt{spear_inst_1}, with $m=512$ samples and $n=1024$ variables,
		and $\lambda = 0.05\lambda_\max$ was used.}
	\label{fig:bpdn_small}
\end{figure}

\begin{table}[tb]
\tiny
\begin{center}
\begin{tabular}{@{}|l|r|r||rr|rr|rr|@{}}
\hline
& & & \multicolumn{2}{c|}{Fast FBS} & \multicolumn{2}{c|}{L-BFGS} & \multicolumn{2}{c|}{Alg.~\ref{alg:GlobalFixed} (L-BFGS)} \\
ID  & $\lambda/\lambda_\max$ & nnz($x_\star$) & iterations & matvecs & iterations & matvecs & iterations & matvecs \\
\hline
\hline
\textbf{191} & $1\cdot 10^{-1}$ & 3  & 2260  & 4520  & 206 & 822 & 105  & 626  \\
$m=8192$ & $5\cdot 10^{-2}$ & 4  & 7529  & 15058 & 403 & 1610 & 195 & 1166 \\
$n=32768$ & $1\cdot 10^{-2}$ & 7  & 11222 & 22444 & 890 & 3558 & 425 & 2546 \\
\hline
\textbf{192} & $9\cdot 10^{-1}$ & 2  & 234   & 468   & 2 & 6 & 2   & 8    \\
$m=8192$ & $7\cdot 10^{-1}$ & 9  & 649   & 1298  & 55 & 218 & 28  & 164  \\
$n=32768$ & $5\cdot 10^{-1}$ & 21 & 2051  & 4102  & 276 & 1102 & 99  & 590  \\
\hline
\textbf{193} & $5\cdot 10^{-1}$ & 3  & 2142  & 4284  & 327 & 1306 & 85  & 506  \\
$m=8192$ & $1\cdot 10^{-1}$ & 4  & 2413  & 4826  & 874 & 3494 & 389 & 2330 \\
$n=32768$ & $5\cdot 10^{-2}$ & 5  & 4761  & 9522  & 1159 & 4634 & 390 & 2336 \\
\hline
\textbf{194} & $9\cdot 10^{-1}$ & 6  & 338   & 676   & 45 & 178 & 15  & 86   \\
$m=8192$ & $7\cdot 10^{-1}$ & 29 & 1532  & 3064  & 400 & 1598 & 132 & 788  \\
$n=32768$ & $5\cdot 10^{-1}$ & 51 & 2482  & 4964  & 616 & 2462 & 294 & 1760 \\
\hline
\textbf{195} & $5\cdot 10^{-2}$ & 2  & 3215  & 6430  & 301 & 1202 & 170 & 1016  \\
$m=8192$ & $1\cdot 10^{-2}$ & 3  & 3833  & 7666  & 685 & 2738 & 352 & 2108 \\
$n=32768$ & $5\cdot 10^{-3}$ & 4  & 10525 & 21050 & 966 & 3862 & 490 & 2936 \\
\hline
\textbf{196} & $7\cdot 10^{-1}$ & 2  & 651   & 1302  & 37 & 146 & 9   & 50   \\
$m=8192$ & $5\cdot 10^{-1}$ & 4  & 1656  & 3312  & 241 & 962 & 16  & 92  \\
$n=32768$ & $1\cdot 10^{-1}$ & 11 & 3075  & 6150  & 515 & 2058 & 265 & 1586 \\
\hline
\textbf{197} & $1\cdot 10^{-1}$ & 3  & 3863  & 7726  & 638 & 2550 & 134 & 800 \\
$m=8192$ & $5\cdot 10^{-2}$ & 4  & 4343  & 8686  & 842 & 3366 & 293 & 1754 \\
$n=32768$ & $1\cdot 10^{-2}$ & 6  & 10894 & 21788 & 2436 & 9742 & 528 & 3164 \\
\hline
\textbf{198} & $9\cdot 10^{-1}$ & 4  & 628   & 1256  & 99 & 394 & 17  & 98   \\
$m=8192$ & $7\cdot 10^{-1}$ & 17 & 2351  & 4702  & 317 & 1266 & 113  & 674  \\
$n=32768$ & $5\cdot 10^{-1}$ & 25 & 3236  & 6472  & 467 & 1866 & 160 & 956  \\
\hline
\textbf{199} & $1\cdot 10^{-1}$ & 2  & 7163  & 14326 & 682 & 2716 & 274 & 1640 \\
$m=8192$ & $5\cdot 10^{-2}$ & 3  & 2537  & 5074  & 1011 & 4042 & 395 & 2366 \\
$n=49152$ & $1\cdot 10^{-2}$ & 5  & 10460 & 20920 & 2124 & 8494 & 770 & 4616 \\
\hline
\textbf{200} & $9\cdot 10^{-1}$ & 6  & 622   & 1244  & 134 & 534 & 25  & 146  \\
$m=8192$ & $7\cdot 10^{-1}$ & 10 & 2144  & 4288  & 297 & 1186 & 86 & 512  \\
$n=49152$ & $5\cdot 10^{-1}$ & 19 & 2313  & 4626  & 538 & 2150 & 197 & 1178 \\
\hline
\end{tabular}
\vspace{5mm}
\caption{Lasso: performance of the algorithms
on the largest of the considered datasets. Column ID indicates the identification number of
the dataset among the SPEAR instances.
Results refer to iterations and matrix vector
products needed to obtain $\varphi(x)-\varphi_\star \leq 10^{-6}|\varphi_\star|$.}\label{tab:BPDN1}
\end{center}
\end{table}

		\subsection{Sparse logistic regression}
			The composite objective function consists of
$$ f(x) = \sum_{i=1}^m\log(1+e^{-b_i\innprod{a_{i}}{x}}),\qquad g(x) = \lambda\|x\|_1.$$
Here vector $a_i\in\Re^n$ contains the features of the $i$-th instance, and $b_i\in\set{}{-1,1}$ indicates the correspondent class.
The $\ell_1$-regularization enforces sparsity in the solution.
We have $\lambda_\max = \tfrac{1}{2}\|A^\T b\|_\infty$, so that for $\lambda \geq \lambda_\max$ the optimal solution is $x_\star = 0$.

We ran the algorithms one three datasets.\footnote{\url{http://www.csie.ntu.edu.tw/~cjlin/libsvmtools/datasets/}}
Indicating with $A\in\Re^{m\times n}$ the matrix having rows $a_1, a_2, \ldots, a_m$,
we computed $\varphi_\star$ numerically, and recorded the number of iterations, calls to $f$, $\nabla f(x)$
and matrix-vector products with $A$, $A^\T$ needed to reach a point $x$ at which
$$ \varphi(x)-\varphi_\star \leq 10^{-8}(1+|\varphi_\star|), $$
The results are shown in \Cref{tab:SparseLogReg1}.

\begin{table}[tb]
  \tiny
  \centering
\begin{tabular}{|l|r|r|r|r|r|}
\hline
& & & \multicolumn{1}{c|}{Fast FBS} & \multicolumn{1}{c|}{L-BFGS} & \multicolumn{1}{c|}{Alg.~\ref{alg:Global} (L-BFGS)} \\
ID  & $\lambda/\lambda_\max$ & nnz($x_\star$) & \multicolumn{1}{c|}{it./$f$/$A$} & \multicolumn{1}{c|}{it./$f$/$A$} & \multicolumn{1}{c|}{it./$f$/$A$} \\
\hline
\textbf{rcv1} & $2\cdot 10^{-1}$ & 25  & 102/209/311       & 130/300/817 & 51/185/436     \\
$m=20242$ & $1\cdot 10^{-1}$ & 70  & 228/461/689       & 271/603/1684 & 92/346/802    \\
$n=44504$ & $5\cdot 10^{-2}$ & 141 & 356/717/1073      & 337/765/2110 & 127/473/1104    \\
$\mbox{nnz}(A)=910K$ & $2\cdot 10^{-2}$ & 287 & 774/1553/2327     & 511/1127/3168 & 210/687/1733    \\
 & $1\cdot 10^{-2}$ & 470 & 943/1891/2834     & 709/1585/4418 & 298/954/2440   \\
\hline
\textbf{real-sim} & $2\cdot 10^{-1}$ & 19  & 65/135/200        & 72/159/444 & 28/92/228     \\
$m=72201$ & $1\cdot 10^{-1}$ & 52  & 204/414/618       & 115/259/716 & 51/176/427     \\
$n=20958$ & $5\cdot 10^{-2}$ & 111 & 332/670/1002      & 161/348/989 & 69/226/567    \\
$\mbox{nnz}(A)=1.5M$ & $2\cdot 10^{-2}$ & 251 & 590/1186/1776     & 237/524/1469 & 113/378/939    \\
 & $1\cdot 10^{-2}$ & 448 & 918/1842/2760     & 339/741/2094 & 153/509/1270    \\
\hline
\textbf{news20} & $2\cdot 10^{-1}$ & 47  & 785/1579/2364     & 298/665/1854 & 118/416/1002    \\
$m=19954$ & $1\cdot 10^{-1}$ & 98  & 1120/2248/3368    & 689/1533/4286 & 203/710/1721  \\
$n=1355191$ & $5\cdot 10^{-2}$ & 208 & 3077/6163/9240    & 703/1592/4401 & 296/1078/2554  \\
$\mbox{nnz}(A)=3.7M$ & $2\cdot 10^{-2}$ & 422 & 6586/13181/19767  & 2114/4881/13334 & 674/2292/5658 \\
 & $1\cdot 10^{-2}$ & 725 & 11072/22153/33225 & 2761/6343/17384 & 1064/3713/9029 \\
\hline
\end{tabular}
\vspace{5mm}
\caption{Sparse logistic regression: detail of the performance of the algorithms
to reach $\varphi(x)-\varphi_\star \leq 10^{-8}|\varphi_\star|$, for different
values of $\lambda$.}
\label{tab:SparseLogReg1}
\end{table}


		\subsection{Group lasso}
			Let vector $x$ be partitioned as $x = (x_1,\ldots,x_N)$, where each $x_i\in\Re^{n_i}$, and $\sum_i n_i = n$.
We consider the $\ell_2$-regularized least squares problem having
$$ f(x) = \frac{1}{2}\| A x-b\|_2^2,\quad g(x) = \lambda\sum_{i=1}^N\|x_i\|_2, $$
where $x = (x_1,\ldots,x_N)$ and $x_i\in\Re^{n_i}$ for $i=1,\ldots,N$.
The $\ell_2$ terms enforce sparsity at the block level, so that for sufficiently large $\lambda$
we expect many of the $x_i$'s to be zero. Partitioning the $A$ by columns as $A=(A_1,\ldots,A_N)$,
with the same block structure at $x$,
then for $\lambda\geq\lambda_\max = \max\set{}{\|A_1^\T b\|_2,\ldots,\|A_N^\T b\|_2}$
the optimal solution is $x_\star = 0$.

To test the methods we generated $A$ and $b$
according to the procedure described in \cite{boyd2011distributed}: take $A$
with normally distributed entries, with zero mean and unit variance,
and $x_{\textrm{true}}$ with a small number of nonzero blocks. Then compute $b = A x_{\textrm{true}}+v$,
where $v$ is a Gaussian noise vector with standard deviation $0.1$.
We compared FBS, L-BFGS and \Cref{alg:Global} with L-BFGS
directions, on a random problem with $m=200$ examples, $N=200$ blocks of variables
and $n_1 = \ldots = n_N = 100$, for a total of $2\cdot 10^4$ decision variables. The resulting
$A$ is a dense matrix with $4$ million coefficients. \Cref{fig:GroupLasso1} shows
the simulation results.

\begin{figure}[tbp]
\centering
	\includegraphics{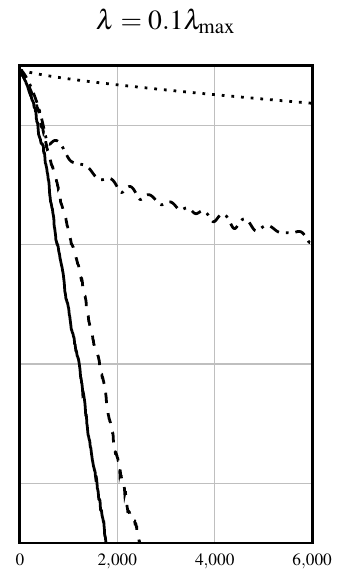}
	\includegraphics{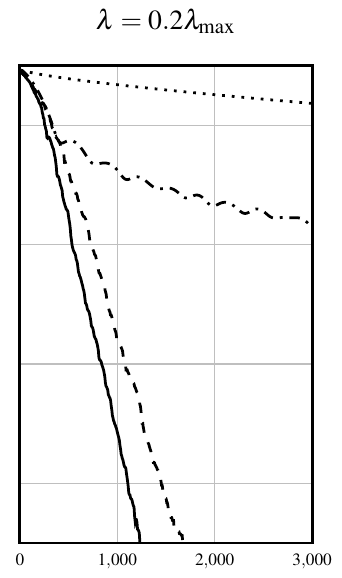}
	\includegraphics{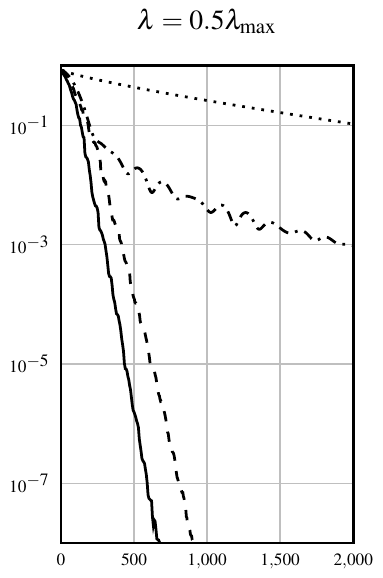}
	
	\includegraphics[width=0.5\linewidth]{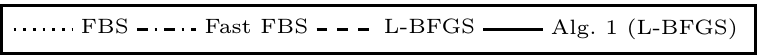}%
	\caption{Group lasso: performance of the proposed method on a random dense problem with
		$m=200$ data points and $n = 2\cdot 10^4$ variables.
		Horizontal axis is the number of matrix-vector products computed,
		vertical axis is the relative error $\|x^k-x_\star\|/\|x_\star\|$.}
	\label{fig:GroupLasso1}
\end{figure}

		\subsection{Matrix completion}
			We consider the problem of recovering the entries of an $m$-by-$n$ matrix, which is known to have small rank, from a sample of them.
One may refer to~\cite{candes2009exact} for a detailed theoretical analysis of the problem.
The decision variable is now a matrix $x = (x_{ij})\in\Re^{m\times n}$, and the problem has the form
$$f(x) = \tfrac{1}{2}\|\mathcal{A}(x)-b\|^2,\qquad g(x) = \lambda\|x\|_*,$$
where $\mathcal{A}:\Re^{m\times n}\to\Re^k$ is a linear mapping selecting $k$ entries from $x$,
vector $b\in\Re^k$ contains the known entries, and $\|x\|_*$ indicates the nuclear norm of $x$, which is the sum of its singular values.
In this case $L_f=1$, therefore we can apply \Cref{alg:GlobalFixed}.

The most computational expensive operation here is the proximal step,
requiring the computation of a SVD.
Computing the full SVD becomes computationally infeasible as $m$ and $n$ grow, therefore we use
the following partial decomposition strategy in computing the proximal mapping: start with $\nu_0 = 10$, and
the $i$-th time $\prox_{\gamma g}$ is evaluated compute only the largest $\nu_i$ singular values $\sigma_1\geq\ldots\geq\sigma_{\nu_i}$, and
$$ \prox_{\gamma g}(x) \approx U \tilde{\Sigma}_+ V^T, \quad \tilde{\Sigma}_+ = \diag(\max\set{}{0, \sigma_i-\gamma\lambda}, i=1,\ldots,\nu_i), $$
Then set $\nu_{i+1}$ according to the following rule
$$ \nu_{i+1} = \begin{cases}
  \min\set{j}{\sigma_j \leq \gamma} & \mbox{if } \sigma_{\nu_i} \leq \gamma\lambda \\
  \nu_i + 5 & \mbox{otherwise.}
\end{cases}$$
The same technique for approximately evaluating the singular value thresholing is used in other algorithms
for nuclear norm regularization problems \cite{toh2010accelerated}.
The partial singular value decompositions were performed using PROPACK software package.\footnote{\url{http://sun.stanford.edu/~rmunk/PROPACK/}}

We compared \Cref{alg:GlobalFixed} with L-BFGS directions, classical L-BFGS and FBS
on the MovieLens100k dataset.\footnote{\url{http://grouplens.org/datasets/movielens/}} This consists of $10^5$ ratings of $1682$ movies
from $943$ users, so that the problem has $\approx 1.6$ millions variables. The results of the simulations, for decreasing values of $\lambda$, are in \Cref{fig:MatComp1}.

\begin{figure}[tb]
\centering
	\includegraphics{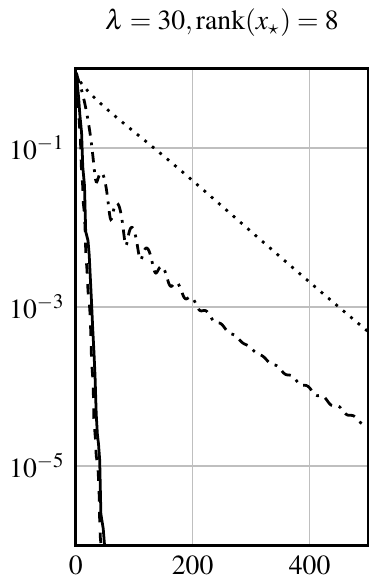}
	\includegraphics{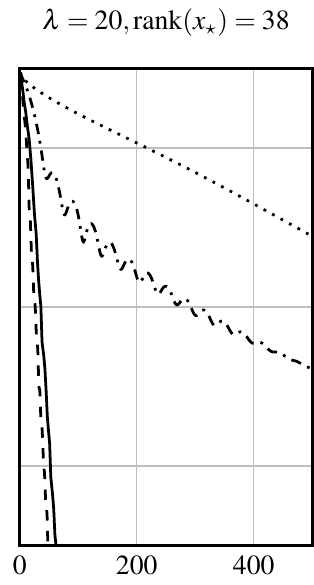}
	\includegraphics{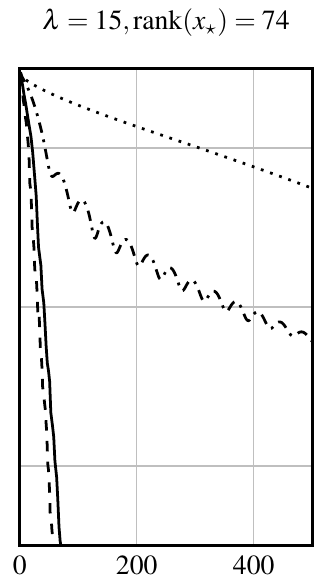}
	
	\includegraphics[width=0.5\linewidth]{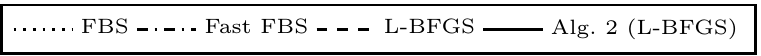}%
	\caption{Matrix completion: performance of \Cref{alg:GlobalFixed} compared with FBS, for different
		values of $\lambda$. Horizontal axis is the number of SVDs computed,
		vertical axis is the relative error $\|x^k-x_\star\|/\|x_\star\|$.}
	\label{fig:MatComp1}
\end{figure}

		\subsection{Image restoration}
			\begin{figure}[p]
	\centering
	\includegraphics{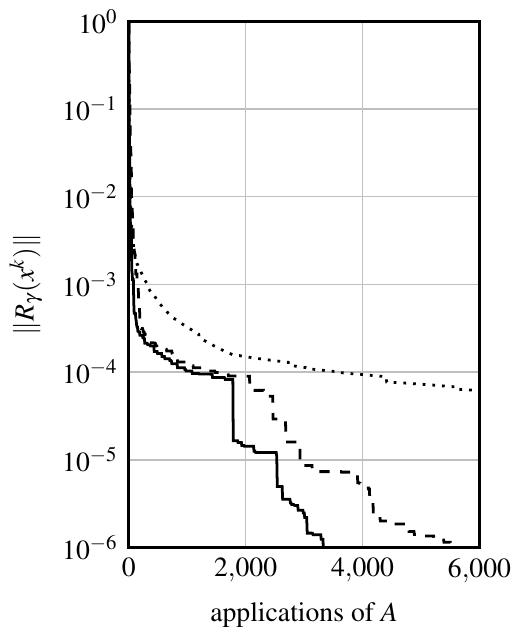}%
	\quad
	\includegraphics{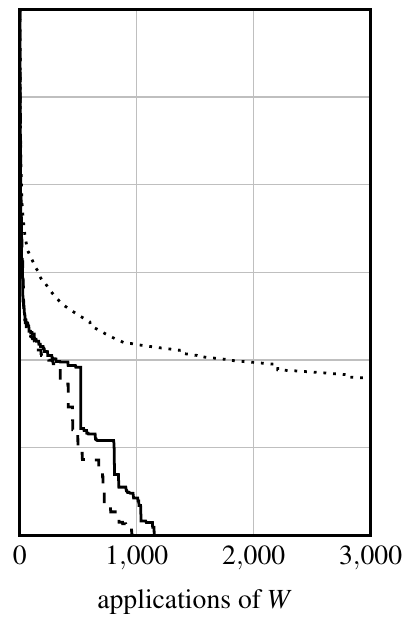}%
	
	\includegraphics[width=0.5\linewidth]{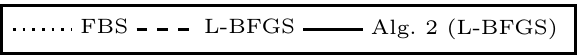}%
	\caption{(Nonconvex) image restoration: performance of \Cref{alg:Global} compared with FBS.
		On the horizontal axis, number of calls to the blur (left plot) and Haar (right plot) operators;
		on the vertical axis the fixed-point residual $R_\gamma$.
		Original, noisy/blurred, and recovered images are shown in \Cref{fig:Elaine}.}
	\label{fig:NoiseBlur}
	
	\vspace{10pt}
	\begin{subfigure}[t]{0.31\linewidth}%
		\caption{Original image}%
		\includegraphics[width=\linewidth]{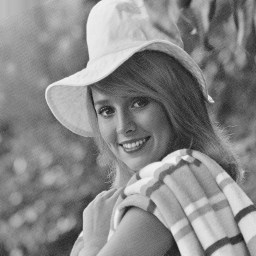}%
	\end{subfigure}%
	\hspace{12pt}%
	\begin{subfigure}[t]{0.31\linewidth}%
		\caption{Noisy blurred image}%
		\includegraphics[width=\linewidth]{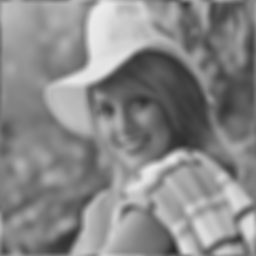}%
	\end{subfigure}
	
	\vspace{9pt}\noindent%
	\begin{subfigure}[t]{0.31\linewidth}%
		\includegraphics[width=\linewidth]{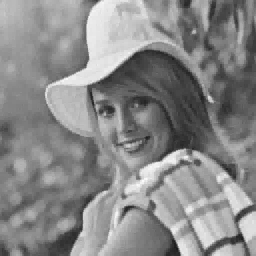}%
		\caption{FBS}%
	\end{subfigure}%
	\hfill%
	\begin{subfigure}[t]{0.31\linewidth}%
		\includegraphics[width=\linewidth]{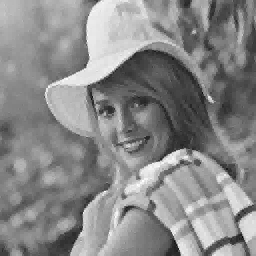}%
		\caption{L-BFGS}%
	\end{subfigure}%
	\hfill%
	\begin{subfigure}[t]{0.31\linewidth}%
		\includegraphics[width=\linewidth]{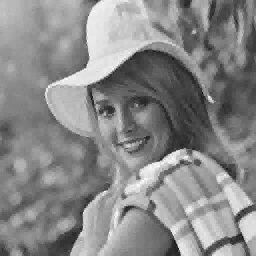}%
		\caption{Alg. 2 (L-BFGS)}%
	\end{subfigure}%
	\caption{(Nonconvex) image restoration: recovered images obtained with the three considered algorithms.}\label{fig:Elaine}
\end{figure}

As a nonconvex example we consider the restoration of a noisy blurred $M\times N$ image. The formulation we use is similar
to that in \cite{boct2014inertial}, although here we consider the $\ell_1$ norm in place of the $\ell_0$ norm as regularization term.
Specifically, we set
$$ f(x) = \sum_{i=1}^{MN} \psi((Ax-b)_i), \qquad g(x) = \lambda\|Wx\|_1. $$ 
Here, $b$ denotes the noisy blurred image, $A$ is a Gaussian blur operator and $W$ is a discrete Haar wavelet transform
with four levels, while $\psi = \log(1+t^2)$, therefore here $\nabla f$ has Lipschitz constant $2$.
Since $W^\top W = WW^\top = I$, the proximal mapping of $g$ can be computed as
\begin{equation}\label{eq:NoiseBlur}
	\prox_{\gamma g}(x) = W^\top \prox_{\gamma\|{}\cdot{}\|_1}(Wx).
\end{equation}
We applied the proposed methods to a $256\times 256$-pixel black-and-white image.
We distorted the original image with a Gaussian blur operator $9\times 9$ with standard deviation $4$, and with a Gaussian noise with standard deviation $10^{-3}$.
The regularization parameter in \eqref{eq:NoiseBlur} was set as $\lambda = 10^{-4}$.
Results of the simulations are shown in \Cref{{fig:Elaine},{fig:NoiseBlur}}.

	\section{Conclusions}
		The forward-backward splitting (FBS) algorithm for minimizing $\varphi = f+g$, where $f$ is smooth and
$g$ is convex, is equivalent to a variable-metric gradient method applied to a continuously differentiable objective, which we called
forward-backward envelope (FBE), when $f\in \cont^2$. Therefore, we can adopt advanced smooth unconstrained
minimization algorithms, such as quasi-Newton and limited-memory methods, to the problem of minimizing the FBE and thus
solving the original, nonsmooth problem. We propose to implement in an algorithmic scheme,
given in \Cref{alg:Global}, which is appealing in that $(i)$ it relies on the very same black-box oracle as FBS
(evaluations of $f$, its gradient, $g$ and its proximal mapping) and is therefore suited for large scale
applications, $(ii)$ it does not require the knowledge of global information such as Lipschitz constant $L_f$,
but can adaptively estimate it.
The proposed method exploits the composite structure of $\varphi$, and alternates line-search steps over descent
directions and forward-backward steps. For this reason, \Cref{alg:Global} possesses the same global convergence
properties of FBS, under the assumptions that $\varphi$ has the Kurdyka-{\L}ojasiewicz properties at its critical points,
and a global convergence rate $O(1/k)$ in case $\varphi$ is convex.
This is a peculiar feature of our approach, since line-search methods do not converge to stationary 
points, in general, when applied to nonconvex functions.
Moreover, we proved that when quasi-Newton directions are used in \Cref{alg:Global}, and the FBE is twice
differentiable with nonsingular Hessian at the limit point of the sequence of iterates,
superlinear asymptotic convergence is achieved.
Our theoretical results are supported by numerical experiments. These show that \Cref{alg:Global} with
(limited-memory) quasi-Newton directions improves the asymptotic convergence of FBS (and its accelerated variant when
$\varphi$ is convex), and usually converges faster than the corresponding classical line-search method applied to
the problem of minimizing the FBE.


\bibliographystyle{IEEEtran}
\bibliography{TeX/FB-QN.bib}


\newcommand\appendixProofTitle[1]{Proofs of \texorpdfstring{\Cref{#1}}{§\ref{#1}}}
\newenvironment{appendixproof}[1]{%
	\par%
	\phantomsection%
	\addcontentsline{toc}{subsection}{Proof of Theorem \ref{#1}}%
	\label{Proof:#1}%
	\begin{proof}[\bf Proof of \Cref{#1}]~
	\par%
}{%
	\end{proof}%
}

\begin{appendix}
	\section{Definitions and known results}
		\label{APP:Definitions}
		Given a differentiable mapping \(G:\Re^n\to\Re^m\) we let \(\jac G{}:\Re^n\to\Re^{m\times n}\) denote the \DEF{Jacobian} of \(G\).
When \(m=1\) we indicate with \(\nabla G=\jac G{}^\top\) the \DEF{gradient} of \(G\) and with \(\nabla^2 G=\jac{\nabla G}{}^\top\) its \DEF{Hessian}, whenever it makes sense.
We say that $G$ is \DEF{strictly differentiable} at $\bar x$ if it satisfies the stronger limit
\[
	\lim_{
		\substack{
			(x,y)\to(\bar x,\bar x)\\
			x\neq y
		}
	}{
		\frac{\|G(y)-G(x)-\jac G{\bar x}[y-x]\|}{\|y-x\|}
	}
{}={}
	0
\]
The next result states that strict differentiability is preserved by composition; its proof is a trivial computation and is therefore omitted.
\begin{prop}\label{prop:StrDiffComp}
Let $F:\Re^n\to\Re^m$, $P:\Re^m\to\Re^k$.
Suppose that $F$ and $P$ are (strictly) differentiable at $\bar x$ and $F(\bar x)$, respectively.
Then the composition $T=P\circ F$ is (strictly) differentiable at $\bar x$.
\end{prop}

Similarly, the product of (strictly) differentiable functions is still (strictly) differentiable.
However, if one of the two functions vanishes at one point, then we may relax some assumptions, as it is proved in the next result.
\begin{prop}\label{prop:StrDiffProd}
Let $Q:\Re^n\to\Re^{m\times k}$ and $R:\Re^n\to\Re^k$, and suppose that $R(\bar x) = 0$.
If $Q$ is continuous at $\bar x$
(resp. Lipschitz-continuous around $\bar x$) and $R$ is differentiable (resp. strictly differentiable) at $\bar x$,
then their product $G:\Re^n \to \Re^m$ defined as $G(x) = Q(x)R(x)$ is differentiable (resp. strictly differentiable) at $\bar x$ with $\jac{G}{\bar x} = Q(\bar x) \jac{R}{\bar x}$.
\begin{proof}
Suppose first that $Q$ is continuous at $\bar x$ and $R$ is differentiable at $\bar x$. Then, expanding \(R(x)\) at \(\bar x\) and since \(G(\bar x)=0\), we obtain
\begin{align*}
	\frac{
		G(x) - G(\bar x) - Q(\bar x)\jac{R}{\bar x}[x-\bar x]
	}{
		\|x-\bar x\|
	}
{}={} & 
	\frac{
		Q(x)R(x)
		{}-{}
		Q(\bar x)\jac{R}{\bar x}[x-\bar x]	
	}{
		\|x-\bar x\|
	}
\\
{}={} &
	\frac{
		\left(
			Q(x)
			{}-{}
			Q(\bar x)
		\right)
		\jac{R}{\bar x}[x-\bar x]
		{}+{}
		o(\|x-\bar x\|)
	}{
		\|x-\bar x\|
	}
\end{align*}
The quantity
\(
	\jac{R}{\bar x}[\frac{x-\bar x}{\|x-\bar x\|}]
\)
is bounded, and continuity of \(Q\) at \(\bar x\) implies that taking the limit for \(\bar x\neq x\to\bar x\) yields \(0\).
This proves that \(G\) is differentiable at \(\bar x\).

Suppose now that $Q$ is Lipschitz-continuous around $\bar x$, and that $R$ is strictly differentiable at $\bar x$.
Then, expanding \(R(y)\) at \(x\) we obtain
\begin{align*}
	\frac{
		G(y) - G(x) - Q(\bar x)\jac{R}{\bar x}[y-x]
	}{
		\|y-x\|
	}
{}={} & 
	\phantom{{}-{}}
	\frac{
		\bigl(
			Q(y)-Q(\bar x)
		\bigr)
		\jac{R}{\bar x}[y-x]
	}{
		\|y-x\|
	}
\\
&
	{}+{}
	\frac{
		\bigl(
			Q(y)-Q(x)
		\bigr)
		R(x)
		{}+{}
		Q(y)o(\|x-y\|)
	}{
		\|y-x\|
	}
\end{align*}
The quantity
\(
	\jac{R}{\bar x}[\vphantom{X^X}\smash{\frac{y-x}{\|y-x\|}}]
\)
is bounded, and by Lipschitz-continuity of \(Q\) at \(\bar x\) so is
\(
	\frac{Q(x)-Q(y)}{\|x-y\|}
\)
for \(x,y\) sufficiently close to \(\bar x\).
Taking the limit for \((\bar x,\bar x)\neq(x,y)\to(\bar x,\bar x)\) with \(x\neq y\) in the above expression then yields \(0\), proving strict differentiability.
Uniqueness of the Jacobian proves also the claimed form of \(\jac G{\bar x}\).
\end{proof}
\end{prop}

\begin{defin}
A mapping $G:\Re^n\to\Re^m$ is said to be \DEF{semidifferentiable} (or \DEF{$B$-differentiable} \cite{pang1990newton,ip1992local}) at a point $\bar x\in\Re^n$ if there exists a positively homogeneous mapping
\(
	\Bjac G{\bar x}[{}\cdot{}]:\Re^n\to\Re^m
\)
such that
\[
	\lim_{x\to\bar x}{
		\frac{
			\|G(x)-G(\bar x)-\Bjac G{\bar x}[x-\bar x]\|
		}{\|x-\bar x\|}
	}
{}={}
	0.
\]
It is \DEF{strictly semidifferentiable} at $\bar x$ if the stronger limit holds
\[
	\lim_{
		\substack{
			(x,y)\to(\bar x,\bar x)\\
			x\neq y
		}
	}{
		\frac{\|G(y)-G(x)-\Bjac G{\bar x}[y-x]\|}{\|y-x\|}
	}
{}={}
	0.
\]
$\Bjac G{\bar x}$ is called \DEF{semiderivative} of $G$ at $\bar x$. If $G$ is (strictly) semidifferentiable
at every point of a set $S$, then it is said to be (strictly) semidifferentiable in $S$.
\end{defin}


\begin{prop}[{\!\!\cite[Thm. 2]{pang1990newton}}]
Suppose that $G:\Re^n\to\Re^m$ is semidifferentiable in a neighborhood of \(\bar x\in\Re^n\).
Then, the following are equivalent:
\begin{enumprop}[{label=(\alph*)},{ref=\theprop(\alph*)}]
	\item\label{Thm:ContSemidiff}
		$\Bjac G{\cdot}[d]$ is continuous in its first argument at \(\bar x\) for all \(d\in\Re^n\);
	\item
		$G$ is strictly semidifferentiable at $\bar x$;
	\item
		$G$ is strictly (Fr\'echet) differentiable at $\bar x$.
\end{enumprop}
\end{prop}

\begin{prop}\label{prop:Calm}
Suppose that $G:\Re^n\to\Re^m$ is semidifferentiable in a neighborhood $N$ of \(\bar x\) and that \(\Bjac G{}\) is \DEF{calm} at $\bar x$, \ie, there exists $L>0$ such that, for all $x\in N$ and $d\in \OurSpace$ with $\|d\|=1$,
\[
	\|\Bjac Gx[d]-\Bjac G{\bar x}[d]\| \leq L\|x-\bar x\|.
\]
Then,
\[
	\|G(x)-G(y)-\Bjac G{\bar x}[x-y]\|
{}\leq{}
	L\max\set{}{\|x-\bar x\|, \|y-\bar x\|}
	\|x-y\|
\]
\begin{proof}
Follows from \cite[Lem. 2.2]{ip1992local} by observing that the assumption of Lipschitz-continuity may be relaxed to calmness.
\end{proof}
\end{prop}

	\section{\appendixProofTitle{SEC:FBE}}
\begin{appendixproof}{lem:DiffFPR}
We know from \cite[Thms. 3.8, 4.1]{poliquin1996generalized} that
$\prox_{\gamma g}$ is (strictly) differentiable at $x-\gamma\nabla f(x)$
if and only if $g$ satisfies \Cref{Ass:GenQuad} (\Cref{Ass:Strict2Epi}) at $x$ for $-\nabla f(x)$.
Since $f\in\cont^2$ by assumption, then in particular $\nabla f$ is strictly differentiable.
The formula \eqref{eq:JacFPR} follows from \Cref{prop:StrDiffComp} with
$P = \prox_{\gamma g}$ and $F(x) = x - \gamma \nabla f(x)$.

Matrix $Q_\gamma(x)$ is symmetric since $f\in\cont^2$ and positive definite if $\gamma < 1/L_f$.
To obtain an expression for $P_\gamma(x) = \jac{\prox_{\gamma g}}{x-\gamma\nabla f(x)}$ we can apply \cite[Ex. 13.45]{rockafellar2011variational} to the \emph{tilted} function $g+\innprod{\nabla f(x)}{{}\cdot{}}$ so that, letting
$\twiceepi g{} =\twiceepi[-\nabla f(x)]gx[{}\cdot{}]$ and \(\proj{S}\) the idempotent and symmetric projection matrix on \(S\),
\begin{align*}
	P_\gamma(x)d
{}={}&
	\prox_{(\gamma/2)\twiceepi g{}}(d)
\\
{}={}&
	\argmin_{d'\in S}{
		\set{}{
			\tfrac12\innprod{d'}{Md'} + \tfrac{1}{2\gamma}\|d' - d\|^2
		}
	}
\\
{}={}&
	\proj{S}
	\argmin_{d'\in\Re^n}{
		\set{}{
			\tfrac12\innprod{\Pi_S d'}{M \Pi_S d'} + \tfrac{1}{2\gamma}\|\Pi_S d' - d\|^2
		}
	}
\\
{}={}&
	\Pi_S
	{\bigl(
		\Pi_S[I+\gamma M]\Pi_S
	\bigr)}^\dagger
	\Pi_S d
\\
{}={}&
	\Pi_S
	[I+\gamma M]^{-1}
	\Pi_S d
\end{align*}
where $^\dagger$ indicates the pseudo-inverse, and last equality is due to \cite[Facts 6.4.12(i)-(ii) and 6.1.6(xxxii)]{bernstein2009matrix} and the properties of $M$ as stated in \Cref{Ass:GenQuad}.
Apparently $P_\gamma(x)\succeq 0$ is symmetric, with $\|P_\gamma(x)\| \leq 1$.
\end{appendixproof}

\begin{appendixproof}{thm:2ndOrder}
If follows from \Cref{thm:DiffGradFBE} that the Hessian $\nabla^2\varphi_\gamma(x)$ exists and is symmetric.
Moreover, from \cite[Ex. 13.18]{rockafellar2011variational} we know that for all $d\in\Re^n$
\begin{align}
\notag
\twiceepi[0]{\varphi}{x}[d] &= \innprod{d}{\nabla^2 f(x)d} + \twiceepi[-\nabla f(x)]{g}{x}[d] \\
\label{Eq:GenQuad}
&= \innprod{d}{\nabla^2 f(x)d} + \innprod{d}{Md} + \delta_S(d).
\end{align}

\ref{it:StrLocMin2} $\Leftrightarrow$ \ref{it:StrLocMin1}:
Follows directly from \eqref{Eq:GenQuad}, using \cite[Thm. 13.24(c)]{rockafellar2011variational}.

\ref{it:StrLocMin3} $\Leftrightarrow$ \ref{it:StrLocMin4}:
Letting $Q = Q_\gamma(x)$, we see from \eqref{eq:JacFPR} and \eqref{eq:HessFBE} that $\jac{R_\gamma}{x}$
is similar to the symmetric matrix $Q^{-1/2}\nabla^2\varphi_\gamma(x)Q^{-1/2}$, which is positive definite if and only if $\nabla^2\varphi_\gamma(x)$ is.

\ref{it:StrLocMin1} $\Leftrightarrow$ \ref{it:StrLocMin3}:
From the point above we know that $\jac{R_\gamma}{x}$ has all real eigenvalues,
and it can be easily seen to be similar to $\gamma^{-1}(I-QP)$, where $P = P_\gamma(x)$.
From \cite[Theorem 7.7.3]{horn2012matrix} it follows that $\lambda_\min(I-QP) > 0$ if and only if $Q^{-1} \succ P$.
For all \(d\in S\), using \eqref{eq:JacP} we have
\begin{align*}
	\innprod{d}{(Q^{-1}-P)d}
{}={} &
	\innprod{d}{Q^{-1}d}
	{}-{}
	\innprod{d}{
		\Pi_S
		[I+\gamma M]^{-1}
		\Pi_Sd
	}
\\
{}={} &
	\innprod{d}{Q^{-1}d}
	{}-{}
	\innprod{\Pi_Sd}{
		[I+\gamma M]^{-1}
		\Pi_Sd
	}
\\
{}={} &
	\innprod{d}{Q^{-1}d}
	{}-{}
	\innprod{d}{
		[I+\gamma M]^{-1}
		d
	}
\end{align*}
and last quantity is positive if and only if $I+\gamma M\succ Q$ on $S$.
By definition of $Q$, we then have that this holds if and only if
$
	\nabla^2 f(x)+M\succ0
$
on $S$, which is \ref{it:StrLocMin1}.

\ref{it:StrLocMin4} $\Leftrightarrow$ \ref{it:StrLocMin5}:
Trivial since $\nabla^2\varphi_\gamma(x)$ exists.
\end{appendixproof}

	\section{\appendixProofTitle{SEC:Algorithms}}
		The following results are instrumental in proving convergence of the iterates of \Cref{alg:Global}.

\begin{lem}\label{le:diffx2Rwx}
Under \Cref{Ass:fg}, consider the sequences $\seq{x^k}$ and $\seq{w^k}$ generated by \Cref{alg:Global}.
If there exist $\bar{\tau},c>0$ such that $\tau_k\leq\bar{\tau}$ and $\|d^k\|\leq c\|R_{\gamma_k}(x^k)\|$, then
\begin{align}
\label{eq:diffx2Rwx}
	\|x^{k+1}-x^k\|
{}\leq{} &
	\gamma_k\|R_{\gamma_k}(w^k)\|
	{}+{}
	\bar{\tau}c\|R_{\gamma_k}(x^k)\|
\qquad
	\forall k\in\Nn
\intertext{and, for \(k\) large enough,}
\label{eq:diffx2Rw}
	\|x^{k+1}-x^k\|
{}\leq{} &
	\gamma_k\|R_{\gamma_k}(w^k)\|
	{}+{}
	\bar{\tau}c(1+\gamma_k L_f)\|R_{\gamma_{k-1}}(w^{k-1})\|
\end{align}
\end{lem}
\begin{proof}
Equation \eqref{eq:diffx2Rwx} follows simply by
$$\|x^{k+1}-x^k\|=
	\|x^{k+1}-w^k+\tau_kd^k\|\leq\gamma_k\|R_{\gamma_k}(w^k)\|+\bar{\tau}c\|R_{\gamma_k}(x^k)\|.$$
Now, for $k$ sufficiently large $\gamma_k = \gamma_{k-1} = \gamma_\infty > 0$,
see \Cref{lem:LowerBoundGamma}, and
\begin{align*}
\|R_{\gamma_{k}}(x^{k})\|&=\gamma_k^{-1}\|x^{k}-T_{\gamma_k}(x^{k})\|\\
	&=		\gamma_k^{-1}\|T_{\gamma_k}(w^{k-1})-T_{\gamma_k}(x^{k})\|\\
	&\leq	\gamma_k^{-1}\|w^{k-1}-\gamma_k\nabla f(w^{k-1})-x^{k}+\gamma_k\nabla f(x^{k})\|\\
	&\leq	\gamma_k^{-1}\|w^{k-1}-x^{k}\|+\|\nabla f(w^{k-1})-\nabla f(x^{k})\|\\
	&\leq	(1+\gamma_k L_f)\|R_{\gamma_{k-1}}(w^{k-1})\|,
\end{align*}
where the first inequality follows from nonexpansiveness of $\prox_{\gamma g}$, and the last one
from Lipschitz continuity of $\nabla f$.
Putting this together with \eqref{eq:diffx2Rwx} gives \eqref{eq:diffx2Rw}.
\end{proof}

\begin{lem}\label{le:finiteLength}
Let $\seq{\beta_k}$ and $\seq{\delta_k}$ be real sequences satisfying $\beta_k\geq 0$, $\delta_k\geq 0$, $\delta_{k+1}\leq\delta_k$ and $\beta_{k+1}^2\leq (\delta_k-\delta_{k+1})\beta_k$ for all $k\in\Nn$. Then $\sum_{k=0}^\infty\beta_k<\infty$.
\end{lem}
\begin{proof}
Taking the square root of both sides in $\beta_{i+1}^2\leq (\delta_i-\delta_{i+1})\beta_i$  and using
$$\sqrt{\zeta\eta}\leq(\zeta+\eta)/2,$$
for any nonnegative numbers $\zeta$, $\eta$, we arrive at $2\beta_{i+1}\leq(\delta_i-\delta_{i+1})+\beta_i$. Summing up the latter for $i=0,\ldots,k$, for any $k\in\Nn$,
\begin{align*}
2\textstyle{\sum_{i=0}^k}\beta_{i+1}&\leq \textstyle{\sum_{i=0}^k}(\delta_i-\delta_{i+1})+\textstyle{\sum_{i=0}^k}\beta_{i}\\
&=\delta_0-\delta_{k+1}+\beta_0-\beta_{k+1}+\textstyle{\sum_{i=0}^k}\beta_{i+1}\\
&\leq\delta_0+\beta_0+\textstyle{\sum_{i=0}^k}\beta_{i+1}.
\end{align*}
Hence
\begin{equation}\sum_{i=0}^\infty\beta_{i+1}\leq\delta_0+\beta_0<\infty,\label{eq:finiteLength}\end{equation}
which concludes the proof.
\end{proof}

\begin{prop}\label{prop:diffx2sum}
Suppose \Cref{Ass:fg} is satisfied and that \(\varphi\) is lower bounded, and consider the sequences generated by \Cref{alg:Global}.
If $\beta\in (0,1)$ and there exist $\bar{\tau},c>0$ such that $\tau_k\leq\bar{\tau}$ and $\|d^k\|\leq c\|R_{\gamma_k}(x^k)\|$ then
\begin{equation}\label{Eq:diffx2sum}
	\sum_{k=0}^\infty\|x^{k+1}-x^k\|^2<\infty.
\end{equation}
If moreover $\seq{x^k}$ is bounded, then
\begin{equation}\label{Eq:omega_props}
	\lim_{k\to\infty}\dist_{\omega(x^0)}(x^k)=0
\end{equation}
and $\omega(x^0)$ is a nonempty, compact and connected subset of $\zer\partial\varphi$ over which $\varphi$ is constant.
\end{prop}
\begin{proof}
\eqref{Eq:diffx2sum} follows from \eqref{eq:diffx2Rwx}, \Cref{{prop:Rxl2},,{prop:Rwl2}}, and the fact that the sum of square-summable sequences is square summable.

If $\seq{x^k}$ is bounded, that $\omega(x^0)$ is nonempty, compact and connected and $\lim_{k\to\infty}\dist_{\omega(x^0)}(x^k)=0$ follow by~\cite[Lem. 5(ii),(iii), Remark 5]{bolte2014proximal}.
That $\varphi$ is constant on $\omega(x^0)$ follows by a similar argument as in~\cite[Lem. 5(iv)]{bolte2014proximal}.
\end{proof}

The following is \cite[Lem. 6]{bolte2014proximal}, therefore we state it with no proof.

\begin{lem}[Uniformized KL property]\label{le:KLuni}
Let $K\subset\Re^n$ be a compact set and suppose that the proper lower semi-continuous function $\varphi:\Re^n\to\Rinf$ is constant on $K$ and satisfies the KL property at every ${x}^\star\in K$. Then there exist $\varepsilon>0$, $\eta>0$, and a continuous concave function  $\psi:[0,\eta]\to[0,+\infty)$ such that properties \ref{Def:KL1}, \ref{Def:KL2} and \ref{Def:KL3} hold, and
\begin{enumlem}[{label=(\roman*')}]
	\setcounter{enumlemi}{3}
	\item for all $x_\star\in K$ and $x$ such that $\dist_K(x)<\varepsilon$ and
		\(
			\varphi(x_\star)< \varphi(x)<\varphi(x_\star)+\eta
		\),
		\begin{equation}\label{eq:KLuni}
			\psi'(\varphi(x)-\varphi(x_\star))\dist(0,\partial\varphi(x))\geq 1.
		\end{equation}
\end{enumlem}
\end{lem}

\begin{appendixproof}{lem:LowerBoundGamma}
Let $\seq{\gamma_k}$ be the sequence of stepsize parameters computed by \Cref{alg:Global}.
To arrive to a contradiction, suppose that $k_0$ is the smallest element of $\Nn$ such that
\[
	\gamma_{k_0} < \min\set{}{\gamma_0, \sigma(1-\beta)/L_f}.
\]
Clearly, $k_0 \geq 1$.
Moreover $\sigma^{-1}\gamma_{k_0}$ must satisfy the condition in step~\ref{step:GlobalCheck}:
for some \(w\in\Re^n\) (corresponding to \(w^k=x^k+\tau_k d^k\) selected before going back to step \ref{step:GlobalFixedStopping} after the condition in step \ref{step:GlobalCheck} is passed, which might differ from the final value of \(w^k\) after step \ref{step:GlobalCheck} is passed)
\[
	\varphi(T_{\sigma^{-1}\gamma_{k_0}}(w))
{}>{}
	\varphi_{\sigma^{-1}\gamma_{k_0}}(w)
	{}-{}
	\frac{\beta\sigma^{-1}\gamma_{k_0}}{2}\|R_{\sigma^{-1}\gamma_{k_0}}(w)\|^2.
\]
But from \Cref{prop:LowBnd} we also have
\begin{align*}
	\varphi(T_{\sigma^{-1}\gamma_{k_0}}(w))
{}\leq{} &
	\varphi_{\sigma^{-1}\gamma_{k_0}}(w)
	{}-{}
	\frac{\sigma^{-1}\gamma_{k_0}}{2}(1-\sigma^{-1}\gamma_{k_0}L_f)\|R_{\sigma^{-1}\gamma_{k_0}}(w)\|^2
\\
{}\leq{} &
	\varphi_{\sigma^{-1}\gamma_{k_0}}(w)
	{}-{}
	\frac{\beta\sigma^{-1}\gamma_{k_0}}{2}\|R_{\sigma^{-1}\gamma_{k_0}}(w)\|^2,
\end{align*}
where last inequality follows from $\sigma^{-1}\gamma_{k_0} < (1-\beta)/L_f$.
This leads to a contradiction, therefore
\(
	\gamma_k
{}\geq{}
	\min\set{}{\gamma_0, \sigma(1-\beta)/L_f}
\)
as claimed.
That \(\gamma_k\) is asymptotically constant follows since the sequence $\seq{\gamma_k}$ is nonincreasing.
\end{appendixproof}

\begin{appendixproof}{prop:GlobalDesc}
We have
\begin{align}
\varphi(x^{k+1}) & \leq \varphi_{\gamma_k}(w^k) - \tfrac{\beta\gamma_k}{2}\|R_{\gamma_k}(w^k)\|^2 \nonumber\\
  & \leq \varphi_{\gamma_k}(x^k) - \tfrac{\beta\gamma_k}{2}\|R_{\gamma_k}(w^k)\|^2 \label{eq:ChainDecr}\\
  & \leq \varphi(x^k)-\tfrac{\beta\gamma_k}{2}\|R_{\gamma_k}(w^k)\|^2-\tfrac{{\gamma_k}}{2}\|R_{\gamma_k}(x^k)\|^2,\nonumber
\end{align}
where the first inequality comes from step~\ref{step:GlobalCheck},
the second from step~\ref{step:GlobalStepsize}
and the third from \Cref{prop:UppBnd}.
This shows \ref{prop:FunDecr}.
Let \(\varphi_\star=\lim_{k\to\infty}\varphi(x^k)\), which exists since \(\seq{\varphi(x^k)}\) is monotone. If \(\varphi_\star=-\infty\), clearly \(\inf\varphi=-\infty\) and \(\omega(x^0)=\emptyset\) due to properness and lower semicontinuity of \(\varphi\) and to the monotonic behavior of \(\seq{\varphi(x^k)}\).
Otherwise, telescoping the inequality we get
\begin{equation}\label{eq:Telescopic}
	\frac 12
	\sum_{i=0}^k{
		\gamma_i
		\left(
			\beta\|R_{\gamma_i}(w^i)\|^2
			{}+{}
			\|R_{\gamma_i}(x^i)\|^2
		\right)
	}
{}\leq{}
	\varphi(x^0)
	{}-{}
	\varphi(x^{k+1})
{}\leq{}
	\varphi(x^0)
	{}-{}
	\varphi_\star
\end{equation}
and since $\gamma_k$ is uniformly lower bounded by a positive number (see \Cref{lem:LowerBoundGamma}) \ref{prop:Rxl2} follows, hence \ref{prop:ClusterCritical}.
If $\beta>0$, observing that for \(k\) large enough such that \(\gamma_k\equiv\gamma_\infty\) we have
\[
	\varphi_{\gamma_k}(w^{k+1})
\smash{{}\stackrel{\raisebox{3pt}{\tiny step~\ref{step:GlobalStepsize}}}{{}\leq{}}{}}
	\varphi_{\gamma_k}(x^{k+1})
\smash{{}\stackrel{\raisebox{3pt}{\tiny step~\ref{step:GlobalFB}}}{\vphantom{\leq}{}={}}{}}
	\varphi_{\gamma_k}(T_k(w^k))
~{}\leq{}~
	\varphi_{\gamma_k}(w^k),
\]
similar argumentations as those for proving \ref{prop:Rxl2} show \ref{prop:Rwl2}.
\end{appendixproof}

\begin{appendixproof}{th:RateGlobal}
If \(\inf\varphi=-\infty\) there is nothing to prove.
Otherwise, since the sequence $\seq{\gamma_k}$ is nonincreasing, from \eqref{eq:Telescopic} we get
$$ \frac{(k+1)\gamma_{k}}{2}\left(\min_{i=0\ldots k}\|R_{\gamma_{i}}(x^i)\|^2 +
  \beta \min_{i=0\ldots k}\|R_{\gamma_{i}}(w^i)\|^2\right)
  \leq \varphi(x^0) - \inf\varphi. $$
Rearranging the terms and invoking \Cref{lem:LowerBoundGamma} gives the result.
\end{appendixproof}

\begin{appendixproof}{th:RateGlobalConv}
The proof is similar to that of~\cite[Thm. 4]{nesterov2013gradient}.
By \Cref{it:bndUpCvx} we know that
$\varphi_\gamma\leq\varphi^\gamma$ for any $\gamma>0$. Combining
this with \eqref{eq:ChainDecr} we get
\begin{equation}\label{eq:MoreauEnvelopeDescent}
	\varphi(x^{k+1})\leq
	\min_{x\in \OurSpace}\set{}{\varphi(x)+\tfrac{1}{2{\gamma_{k}}}\|x-x^k\|^2},
\end{equation}
and in particular, for $x_\star\in \argmin\varphi$,
\begin{align*}
	\varphi(x^{k+1})
{}\leq{} &
	\min_{\alpha\in[0,1]}{
		\set{}{
			\varphi(\alpha x_\star+(1-\alpha)x^k)+\tfrac{\alpha^2}{2{\gamma_{k}}}\|x^k-x_\star\|^2
		}
	}
\\
{}\leq{} &
	\min_{\alpha\in[0,1]}{
		\set{}{
			\varphi(x^k)-\alpha(\varphi(x^k)-\inf\varphi)+\tfrac{R^2}{2{\gamma_{k}}}\alpha^2
		}
	},
\end{align*}
where the last inequality follows by convexity of $\varphi$.
If $\varphi(x^0)-\inf\varphi\geq R^2/{\gamma_0}$, then the optimal solution of the latter problem for $k=0$ is $\alpha=1$ and we obtain~\eqref{eq:FirstStep}.
Otherwise, the optimal solution is
$$\alpha=\frac{{\gamma_{k}}(\varphi(x^k)-\inf\varphi)}{R^2}\leq \frac{{\gamma_{k}}(\varphi(x^0)-\inf\varphi)}{R^2}\leq 1,$$
and we obtain
$$\varphi(x^{k+1})\leq \varphi(x^k)-\frac{{\gamma_{k}}(\varphi(x^k)-\inf\varphi)^2}{2R^2}.$$
Letting $\lambda_k=\frac{1}{\varphi(x^k)-\inf\varphi}$ the latter inequality is expressed as
$$\frac{1}{\lambda_{k+1}}\leq\frac{1}{\lambda_k}-\frac{{\gamma_{k}}}{2R^2\lambda_{k+1}^2}.$$
Multiplying both sides by $\lambda_{k}\lambda_{k+1}$ and rearranging
\begin{align*}
\lambda_{k+1}\geq\lambda_k+\frac{{\gamma_{k}}}{2R^2}\frac{\lambda_{k+1}}{\lambda_k}\geq \lambda_k+\frac{{\gamma_{k}}}{2R^2},
\end{align*}
where the latter inequality follows from the fact that $\seq{\varphi(x^{k})}$ is nonincreasing,
cf. \Cref{prop:FunDecr}.
Telescoping the inequality and using \Cref{lem:LowerBoundGamma}, we obtain
$$ \lambda_k\geq\lambda_0+\frac{k\min\set{}{\gamma_0, \sigma(1-\beta)/L_f}}{2R^2}
  \geq\frac{k\min\set{}{\gamma_0, \sigma(1-\beta)/L_f}}{2R^2}. $$
Rearranging, we arrive at~\eqref{eq:kStep}.
\end{appendixproof}

\begin{appendixproof}{thm:LocalLinConvCVX}
If \eqref{eq:SOSCphi} holds, then $\varphi$ has bounded level sets and \(\zer\partial\varphi=\set{}{x_\star}\).
In particular, \(\omega(x^0)\neq\emptyset\) and \Cref{prop:ClusterCritical} then ensures \(x^k\to x_\star\).
Therefore, there is $k_0\in\Nn$ such that $x^k \in N$ for all $k\geq k_0$.
Inequality \eqref{eq:MoreauEnvelopeDescent} holds, and in particular for $k\geq k_0$
\begin{align*}
	\varphi(x^{k+1})
{}\leq{} &
	\min_{\alpha\in[0,1]}{
		\set{}{
			\varphi(\alpha x_\star+(1-\alpha)x^k)+\tfrac{\alpha^2}{2{\gamma_k}}\|x_\star - x^k\|^2
		}
	}
\\
{}\leq{} &
	\min_{\alpha\in[0,1]}{
		\set{}{
			\varphi(x^k)+\alpha\left(\tfrac{\alpha}{c{\gamma_k}}-1\right)(\varphi(x^k)-\inf\varphi)
		}
	},
\end{align*}
where the second inequality follows by convexity of $\varphi$ and \eqref{eq:SOSCphi}.
The minimum of last expression is achieved for $\alpha = \min\set{}{1,\tfrac{c}{2}{\gamma_k}}$.
When $\gamma_k < 2c^{-1}$ we have the bound
$$ \varphi(x^{k+1}) - \inf\varphi \leq (1-\tfrac{c}{4}\gamma_k)(\varphi(x^k) - \inf\varphi). $$
When instead $\gamma_k\geq 2c^{-1}$ we have the bound
$$ \varphi(x^{k+1}) - \inf\varphi \leq (c\gamma_k)^{-1}(\varphi(x^k) - \inf\varphi). $$
\Cref{lem:LowerBoundGamma} ensures $\gamma_k \geq \min\set{}{\gamma_0,\sigma(1-\beta)/L_f}$,
so $\max\set{}{1-\tfrac{c}{4}\gamma_k, (c\gamma_k)^{-1}} \leq \omega$.
This proves the claim on the sequence $\seq{\varphi(x^k)}[k\geq k_0]$ and using inequality \eqref{eq:ChainDecr}
the same holds for $\seq{\varphi_{\gamma_k}(w^k)}[k\geq k_0]$. From the error bound \eqref{eq:SOSCphi} we
obtain that $x^k\to x_\star$ linearly.
If the same error bound holds for $\varphi_{\gamma_\infty}$, then also $w^k\to x_\star$ linearly.
\end{appendixproof}

\begin{appendixproof}{thm:Convergence}
The case where the sequence is finite does not deserve any further investigation, therefore we assume that $\seq{x^k}$ is infinite. We then assume that $R_{\gamma_k}(x^k)\neq 0$ which implies through \Cref{prop:GlobalDesc} that $\varphi(x^{k+1})<\varphi(x^k)$.
Due to \eqref{Eq:omega_props}, the KL property for $\varphi$, and \Cref{le:KLuni}, there exist  $\varepsilon,\eta>0$ and a continuous concave function  $\psi:[0,\eta]\to[0,+\infty)$ such that for all $x$ with $\dist_{\omega(x^0)}(x)<\varepsilon$ and $\varphi({x}^\star)< \varphi(x)<\varphi(x_\star)+\eta$ one has
\[
	\psi'(\varphi(x)-\varphi(x_\star))\dist(0,\partial\varphi(x))\geq 1.
\]
According to \Cref{prop:diffx2sum} there exists a $k_1\in\Nn$ such that $\dist_{\omega(x^0)}(x^k)<\varepsilon$ for all $k\geq k_1$. Furthermore, since $\varphi(x^k)$ converges to $\varphi(x_\star)$ there exists a $k_2$ such that $\varphi(x^k)<\varphi(x_\star)+\eta$ for all $k\geq k_2$. Take $\bar{k}=\max\set{}{k_1,k_2}$. Then for every $k\geq\bar{k}$ we have
\[
	\psi'(\varphi(x^k)-\varphi(x_\star))\dist(0,\partial\varphi(x^k))\geq 1.
\]
From \Cref{prop:FunDecr}
\[
	\varphi(x^{k+1})\leq \varphi(x^k)-\tfrac{\beta\gamma_k}{2}\|R_{\gamma_k}(w^k)\|^2.
\]
For every $k>0$ let
\(
	\tilde{\nabla}\varphi(x^k)
{}={}
	\nabla f(x^{k})-\nabla f(w^{k-1})+R_{\gamma_{k-1}}(w^{k-1})
\).
Since
\(
	R_{\gamma_{k-1}}(w^{k-1})
{}\in{}
	\nabla f(w^{k-1}) + \partial g(x^k)
\),
then
\(
	\tilde{\nabla}\varphi(x^k)
{}\in{}
	\partial\varphi(x^k)
\)
and
\begin{align*}
	\|\tilde{\nabla}\varphi(x^k)\|
{}\leq{} &
	\|\nabla f(x^{k})-\nabla f(w^{k-1})\|
	{}+{}
	\|R_{\gamma_{k-1}}(w^{k-1})\|
\\
{}={} &
	(1+{\gamma_{k-1}} L_f)
	\|R_{\gamma_{k-1}}(w^{k-1})\|.
\end{align*}
From \eqref{eq:KLuni}
\[
	\psi'(\varphi(x^{k})-\varphi(x_\star))
{}\geq{}
	\frac{1}{\|\tilde{\nabla}\varphi(x^k)\|}
{}\geq{}
	\frac{1}{(1+{\gamma_{k-1}} L_f)
	\|R_{\gamma_{k-1}}(w^{k-1})\|}.
\]
Let $\Delta_k= \psi(\varphi(x^{k})-\varphi(x_\star))$.
By concavity of $\psi$ and \Cref{prop:FunDecr}
\begin{align*}
\Delta_k-\Delta_{k+1}&\geq\psi'(\varphi(x^{k})-\varphi(x_\star))(\varphi(x^k)-\varphi(x^{k+1}))\\
&\geq \frac{\beta\gamma_k}{2 (1+\gamma_{k-1} L_f)}\frac{\|R_{\gamma_{k}}(w^k)\|^2}{\|R_{\gamma_{k-1}}(w^{k-1})\|}\\
&\geq \frac{\beta\gamma_{\min}}{2 (1+\gamma_0 L_f)}\frac{\|R_{\gamma_{k}}(w^k)\|^2}{\|R_{\gamma_{k-1}}(w^{k-1})\|}
\end{align*}
where $\gamma_\min = \min\set{}{\gamma_0, \sigma(1-\beta)/L_f}$, see \Cref{lem:LowerBoundGamma}, or
\begin{equation}\|R_{\gamma_{k}}(w^{k})\|^2\leq \alpha(\Delta_k-\Delta_{k+1})\|R_{\gamma_{k-1}}(w^{k-1})\|
\label{eq:wResConvergence}\end{equation}
where $\alpha= 2 (1+\gamma_0 L_f)/(\beta\gamma_{\min})$.
Applying \Cref{le:finiteLength} with
$$\delta_k=\alpha\Delta_k,\quad\beta_k=\|R_{\gamma_{k-1}}(w^{k-1})\|,$$
we conclude that $\sum_{k=0}^\infty\|R_{\gamma_{k}}(w^{k})\|<\infty$.
From~\eqref{eq:diffx2Rw}, using the fact that $\gamma_k\leq \gamma_0$ for all $k$,
then it follows that
$$\sum_{k=0}^{\infty}\|x^{k+1}-x^k\|<\infty.$$
Then $\seq{x^k}$ is a Cauchy sequence, hence it converges to a point
that, by \Cref{prop:GlobalDesc}, is a critical point $x_\star$ of $\varphi$.
\end{appendixproof}

\begin{appendixproof}{thm:LocalLinConvKL}
\Cref{thm:Convergence} ensures that $\seq{x^k}$ converges to a critical point, be it $x_\star$.
We know from \Cref{lem:LowerBoundGamma} that eventually $\gamma_k = \gamma_\infty > 0$,
therefore we assume $k$ is large enough for this purpose and indicate $\gamma$
in place of $\gamma_k$ for simplicity. Denoting $A_k=\sum_{i=k}^\infty\|x^{i+1}-x^i\|$
clearly $A_k \geq \|x^k-x_\star\|$,
so we will prove that $A_k$ converges linearly to zero to obtain the result.
Note that by \eqref{eq:diffx2Rw} we know that
$$ \|x^{i+1}-x^i\| \leq \gamma\|R_{\gamma}(w^i)\| +
	\bar{\tau} c (1+\gamma L_f)\|R_{\gamma}(w^{i-1})\|. $$
Therefore we can upper bound $A_k$ as follows
\begin{align}
	\nonumber
	A_k
{}\leq{} &
	\textstyle
	\bar{\tau} c (1+\gamma L_f)\|R_{\gamma}(w^{k-1})\|
	{}+{}
	\left(
		\gamma + \bar{\tau}c(1+\gamma L_f)
	\right)
	\sum_{i=k}^\infty{
		\|R_{\gamma}(w^i)\|
	}
\\
{}\leq{} &
	\textstyle
	\left(
		\gamma + \bar{\tau}c(1+\gamma L_f)
	\right)
	\sum_{i=k-1}^\infty{
		\|R_{\gamma}(w^i)\|,
	}
	\label{eq:uppBndA}
\end{align}
and reduce the problem to proving linear convergence of $B_k = \sum_{i=k}^\infty \|R_{\gamma}(w^i)\|$.
When $\psi$ is as in \eqref{eq:LojasiewiczProperty}, for sufficiently large
$k$ the KL inequality reads
\begin{equation*}
  \varphi(x^k)-\varphi(x_\star) \leq [\sigma(1-\theta)\|v^k\|]^{\frac{1}{\theta}},\quad\forall v^k\in\partial\varphi(x^k).
\end{equation*}
Taking $v^k = \nabla f(x^k) - \nabla f(w^{k-1}) + R_{\gamma}(w^{k-1}) \in \partial\varphi(x^k)$,
this in turn yields
\begin{equation}
  \varphi(x^k)-\varphi(x_\star) \leq
    \left[\sigma(1-\theta)(1+{\gamma} L_f)\|R_{\gamma}(w^{k-1})\|\right]^{\frac{1}{\theta}},
  \label{eq:LojasiewiczIneq}
\end{equation}
(see the proof of \Cref{thm:Convergence}). Inequality \eqref{eq:wResConvergence} holds,
for sufficiently large $k$, with
$\Delta_k = \sigma(\varphi(x^k)-\varphi(x_\star))^{1-\theta}$ in this case.
Applying \Cref{le:finiteLength} with
$$\delta_k=\alpha\Delta_{k},\quad\beta_k=\|R_{\gamma}(w^{k-1})\|=B_{k-1}-B_k,$$
we obtain
\begin{align*}
B_k	&\leq (B_{k-1} - B_k) + \sigma(\varphi(x^k)-\varphi(x_\star))^{1-\theta} \\
  	&\leq (B_{k-1} - B_k) + \sigma\left[\sigma(1-\theta)(1+{\gamma} L_f)(B_{k-1} - B_k)\right]^{\frac{1-\theta}{\theta}},
\end{align*}
where the second inequality is due to \eqref{eq:LojasiewiczIneq}. Since
$B_{k-1}-B_k \to 0$, then for $k$ large enough it holds that $\sigma (1+{\gamma} L_f)(B_{k-1}-B_k) \leq 1$,
and the last term in the previous chain of inequalities is increasing in
$\theta$ when $\theta\in (0,\tfrac{1}{2}]$. Therefore $B_k$ eventually satisfies
$$ B_k \leq C(B_{k-1}-B_k), $$
where $C>0$, and so $B_k \leq [C/(1+C)] B_{k-1}$, \ie, $B_k$ converges
to zero $Q$-linearly. This in turn implies that $\|x^k-x_\star\|$ converges to zero with $R$-linear rate.
Furthermore,
\begin{align*}
\|w^k-x_\star\| &= \|x^k - x_\star + \tau_k d^k\| \\
	&\leq \|x^k-x_\star\| + \bar{\tau}c\|R_{\gamma_k}(x^k)\| \\
	&=    \|x^k-x_\star\| + \bar{\tau}c{\gamma_k}^{-1}\|T_{\gamma_k}(x^k)-x^k\| \\
	&\leq (1+\bar{\tau}c\gamma_k^{-1})\|x^k-x_\star\| + \bar{\tau}c{\gamma_k}^{-1}\|T_{\gamma_k}(x^k)-T_{\gamma_k}(x_\star)\|\\
	&\leq (1+\bar{\tau}c\gamma_k^{-1})\|x^k-x_\star\| + \bar{\tau}c{\gamma_k}^{-1}\|x^k - \gamma_k\nabla f(x^k)- x_\star + \gamma_k\nabla f(x_\star)\| \\
	&\leq (1+\bar{\tau}c(2\gamma_k^{-1} + L_f))\|x^k-x_\star\|,
\end{align*}
where the last two inequalities follow by nonexpansiveness of $\prox_{\gamma g}$ and Lipschitz continuity of $\nabla f$.
Since $\gamma_k$ is lower bounded by a positive quantity, then we deduce that also $\|w^k-x_\star\|$ converges $R$-linearly
to zero.
\end{appendixproof}



	\section{\appendixProofTitle{SEC:QuasiNewton}}
\begin{appendixproof}{thm:SuperlinearConvergence1}
Since $w^k = x^k - B_k^{-1}\nabla\varphi_\gamma(x^k)$, letting \(k\to\infty\) and using \eqref{eq:DennisMoreCondition} we have that
\begin{align*}
	0
{}\leftarrow{}
	\frac{(B_k - \nabla^2\varphi_\gamma(x_\star))(w^k-x^k)}{\|w^k-x^k\|}
{}={} &
	{}-\frac{\nabla\varphi_\gamma(x^k) + \nabla^2\varphi_\gamma(x_\star)(w^k-x^k)}{\|w^k-x^k\|}
\\
{}={} &
	{}-\frac{\nabla\varphi_\gamma(x^k) - \nabla\varphi_\gamma(w^k) + \nabla^2\varphi_\gamma(x_\star)(w^k-x^k)}{\|w^k-x^k\|}
\\
 &
	{}-\frac{\nabla\varphi_\gamma(w^k)}{\|w^k-x^k\|}.
\end{align*}
By strict differentiability of $\nabla\varphi_\gamma$ at $x_\star$ we obtain
\begin{equation}\label{eq:LimIters}
	\lim_{k\to\infty}{
		\frac{
			\|\nabla\varphi_\gamma(w^k)\|
		}{
			\|w^k-x^k\|
		}
	}
{}={}
	0
\end{equation}
By nonsingularity of $\nabla^2\varphi_\gamma(x_\star)$ and since \(w^k\to x^\star\), there exist $\alpha>0$ such that
\(
	\|\nabla\varphi_\gamma(x^k)\|\geq\alpha\|x^k-x_\star\|
\)
for \(k\) large enough.
Therefore, for $k$ sufficiently large,
\[
	\frac{
		\|\nabla\varphi_\gamma(w^k)\|
	}{
		\|w^k-x^k\|
	}
{}\geq{}
	\frac{
		\alpha\|w^k-x_\star\|
	}{
		\|w^k-x^k\|
	}
{}\geq{}
	\frac{
		\alpha\|w^k-x_\star\|
	}{
		\|w^k-x_\star\|+\|x^k-x_\star\|
	}.
\]
Using \eqref{eq:LimIters} we get
\[
	\lim_{k\to\infty}{
		\frac{
			\|w^k-x_\star\|
		}{
			\|w^k-x_\star\|+\|x^k-x_\star\|
		}
	}
{}={}
	\lim_{k\to\infty}{
		\frac{
			\|w^k-x_\star\|/\|x^k-x_\star\|
		}{
			\|w^k-x_\star\|/\|x^k-x_\star\|+1
		}
	}
{}={}
	0,
\]
from which we obtain
\begin{equation}\label{eq:wxIterates}
	\lim_{k\to\infty}{
		\frac{\|w^k-x_\star\|}{\|x^k-x_\star\|}
	}
{}={}
	0.
\end{equation}
Finally,
\begin{align}
\nonumber
	\|x^{k+1}-x_\star\|
{}={} &
	\|T_\gamma(w^k)-T_\gamma(x_\star)\|
\\
\nonumber
{}={} &
	\left\|
		\prox_{\gamma g}(w^k-\gamma\nabla f(w^k))
		{}-{}
		\prox_{\gamma g}(x_\star-\gamma\nabla f(x_\star))
	\right\|
\\
\nonumber
{}\leq{} &
	\left\|
		w^k - \gamma\nabla f(w^k)
		{}-{}
		x_\star + \gamma\nabla f(x_\star)
	\right\|
\\
\label{eq:xkBound}
{}\leq{} &
	(1+\gamma L_f)\|w^k-x_\star\|,
\end{align}
where the first inequality follows from nonexpansiveness of $\prox_{\gamma g}$ and the second
from Lipschitz continuity of $\nabla f$. Using \eqref{eq:xkBound} in \eqref{eq:wxIterates}
we obtain that $\seq{x^k}$ converges to $x_\star$ $Q$-superlinearly, and
$\seq{w^k}$ converges $R$-superlinearly.
\end{appendixproof}

\begin{appendixproof}{thm:SuperlinearConvergence2}
From \Cref{Thm:ContSemidiff} it follows that $\nabla\varphi_\gamma$ is strictly differentiable and continuosly semidifferentiable at \(x_\star\).
Moreover, we know from \Cref{lem:LowerBoundGamma} that eventually $\gamma_k = \gamma_\infty > 0$.
Therefore we assume that $k$ is large enough for this purpose and indicate $\gamma$
in place of $\gamma_k$ for simplicity.
We denote for short $g^k = \nabla\varphi_\gamma(x^k)$.
In \Cref{alg:Global}
$$ w^k-x^k = \tau_k d^k = -\tau_k B_k^{-1}g^k, $$
and by \eqref{eq:DennisMoreCondition} and Cauchy-Schwarz inequality
\begin{align}
\frac{\|(B_k-\nabla^2\varphi_\gamma(x_\star))(w^k-x^k)\|}{\|w^k-x^k\|} &= 
\frac{\|g^k+\nabla^2\varphi_\gamma(x_\star)d^k\|}{\|d^k\|} \nonumber \\
&\geq
	\left|\frac{\innprod{d^k}{g^k+\nabla^2\varphi_\gamma(x_\star)d^k}}{\|d^k\|^2}\right| \to 0. \nonumber
\end{align}
Therefore
\begin{equation} -\innprod{g^k}{d^k} = \innprod{d^k}{\nabla^2\varphi_\gamma(x_\star)d^k} + o(\|d^k\|^2). \label{eq:SecondOrderTermApprox}\end{equation}
Since $\nabla^2\varphi_\gamma(x_\star)$ is positive definite, then there is $\eta>0$ such that
for sufficiently large $k$
\begin{equation} - \innprod{g^k}{d^k} \geq \eta\|d^k\|^2. \label{eq:Idontknow}\end{equation}
Since $\Bjac{\nabla\varphi_\gamma}{}$ is continuous at $x_\star$ and $x^k\to x_\star$, we have
\begin{equation}\label{eq:ContSemider}
\|\Bjac{\nabla\varphi_\gamma}{x^k}[d^k] - \nabla^2 \varphi_\gamma(x_\star)d^k\| = o(\|d^k\|).
\end{equation}
Next, since $x^k\to x_\star$, for $k$ large enough $\nabla\varphi_\gamma$ is semidifferentiable at $x^k$ and we can expand $\varphi_\gamma$ around $x^k$ using \cite[Ex. 13.7(c)]{rockafellar2011variational} to obtain
\begin{align*}
\varphi_\gamma(x^k+d^k) - \varphi_\gamma(x^k) &= \innprod{g^k}{d^k}
		+ \tfrac{1}{2}\innprod{d^k}{\Bjac{\nabla\varphi_\gamma}{x^k}[d^k]} + o(\|d^k\|^2) \\
	&= \innprod{g^k}{d^k} + \tfrac{1}{2}\innprod{d^k}{\nabla^2\varphi_\gamma(x_\star)d^k} + o(\|d^k\|^2) \\
	&= \tfrac{1}{2}\innprod{g^k}{d^k} + o(\|d^k\|^2),
\end{align*}
where the second equality is due to \eqref{eq:ContSemider}, and the last equality is due to \eqref{eq:SecondOrderTermApprox}. Therefore, using \eqref{eq:Idontknow}, for sufficiently large $k$
\begin{align*}
\varphi_\gamma(x^k+d^k) - \varphi_\gamma(x^k) \leq -\tfrac{\eta}{2}\|d^k\|^2 < 0.
\end{align*}
\ie, $\tau_k = 1$ satisfies the non-increase condition.
As a consequence, \Cref{alg:Global} eventually reduces to the iterations of
\Cref{thm:SuperlinearConvergence1} and the proof follows.
\end{appendixproof}

\begin{appendixproof}{thm:SuperlinearConvergenceBFGS}
Suppose that \Cref{ass:Convex} holds. Since $x_\star\in\zer\partial\varphi$
and $\nabla^2\varphi_\gamma(x_\star) \succ 0$,
it follows that $x_\star$ is a strong local minimizer of $\varphi_\gamma$, hence of \(\varphi\) in light of \Cref{{prop:UppBnd},,{prop:FBEzer}}.
\Cref{thm:LocalLinConvCVX} then ensures that $\seq{x^k}$ and $\seq{w^k}$ converge linearly to $x_\star$.
If instead $\seq{\|B_k^{-1}\|}$ is bounded and \Cref{ass:KL} holds, then \Cref{thm:LocalLinConvKL} applies and again $\seq{x^k}$ and $\seq{w^k}$ converge linearly to a critical point, be it $x_\star$.
In both cases we can apply \Cref{prop:Calm} and for $k$ sufficiently large
\begin{equation}\label{eq:Summable}
\frac{\|y^k-\nabla^2\varphi_\gamma(x_\star)s^k\|}{\|s^k\|} \leq L\max\set{}{\|w^k-x\|,\|x^k-x\|}.
\end{equation}
Since the convergence is linear, then the right-hand side of \eqref{eq:Summable} is summable.
With similar arguments to those of \cite[Lem. 3.2]{dennis1974characterization}
we can see that eventually $\innprod{s^k}{y^k}>0$.
Therefore we can apply \cite[Thm. 3.2]{byrd1989tool}, which ensures that
condition \eqref{eq:DennisMoreCondition} holds. The result follows then from \Cref{thm:SuperlinearConvergence2}.
\end{appendixproof}

\end{appendix}

\end{document}